\documentclass[notitlepage,11pt,reqno]{amsart}
\usepackage[foot]{amsaddr}
\usepackage{amssymb,nicefrac,bm,upgreek,mathtools,verbatim,enumerate}
\usepackage[final]{hyperref}
\usepackage[mathscr]{eucal}
\usepackage{dsfont}
\usepackage[normalem]{ulem}
\usepackage{amsopn}

\usepackage[margin=1in]{geometry}
\newcommand{\ttup}[1]{\textup{(}#1\textup{)}}
\newcommand{\stkout}[1]{\ifmmode\text{\sout{\ensuremath{#1}}}\else\sout{#1}\fi}

\newtheorem{lemma}{Lemma}[section]
\newtheorem{theorem}{Theorem}[section]

\newtheorem{corollary}{Corollary}[section]

\theoremstyle{definition}
\newtheorem{definition}{Definition}[section]
\newtheorem{assumption}{Assumption}[section]

\newtheorem{example}{Example}[section]

\theoremstyle{remark}
\newtheorem{remark}{Remark}[section]
\numberwithin{theorem}{section}
\numberwithin{equation}{section}

\usepackage[utf8]{inputenc}
\usepackage{apptools}
\AtAppendix{\counterwithin{lemma}{section}}

\hypersetup{
  colorlinks=true,
 citecolor=mblue,
 linkcolor=mblue,
 frenchlinks=false,
 pdfborder={0 0 0},
 naturalnames=false,
 hypertexnames=false,
 breaklinks}
\usepackage[capitalize,nameinlink]{cleveref}

\usepackage[abbrev,msc-links,nobysame]{amsrefs} 

\crefname{section}{Section}{Sections}
\crefname{subsection}{Section}{Sections}
\crefname{condition}{Condition}{Conditions}
\crefname{hypothesis}{Hypothesis}{Conditions}
\crefname{assumption}{Assumption}{Assumptions}
\crefname{lemma}{Lemma}{Lemmas}
\crefname{fact}{Fact}{Facts}

\Crefname{figure}{Figure}{Figures}

\crefformat{equation}{\textup{#2(#1)#3}}
\crefrangeformat{equation}{\textup{#3(#1)#4--#5(#2)#6}}
\crefmultiformat{equation}{\textup{#2(#1)#3}}{ and \textup{#2(#1)#3}}
{, \textup{#2(#1)#3}}{, and \textup{#2(#1)#3}}
\crefrangemultiformat{equation}{\textup{#3(#1)#4--#5(#2)#6}}%
{ and \textup{#3(#1)#4--#5(#2)#6}}{, \textup{#3(#1)#4--#5(#2)#6}}%
{, and \textup{#3(#1)#4--#5(#2)#6}}

\Crefformat{equation}{#2Equation~\textup{(#1)}#3}
\Crefrangeformat{equation}{Equations~\textup{#3(#1)#4--#5(#2)#6}}
\Crefmultiformat{equation}{Equations~\textup{#2(#1)#3}}{ and \textup{#2(#1)#3}}
{, \textup{#2(#1)#3}}{, and \textup{#2(#1)#3}}
\Crefrangemultiformat{equation}{Equations~\textup{#3(#1)#4--#5(#2)#6}}%
{ and \textup{#3(#1)#4--#5(#2)#6}}{, \textup{#3(#1)#4--#5(#2)#6}}%
{, and \textup{#3(#1)#4--#5(#2)#6}}

\crefdefaultlabelformat{#2\textup{#1}#3}
\newcommand{\vertiii}[1]{{\left\vert\kern-0.25ex\left\vert\kern-0.25ex\left\vert #1 
    \right\vert\kern-0.25ex\right\vert\kern-0.25ex\right\vert}}

\newcommand{\cA}{{\mathcal{A}}}  
\newcommand{\sA}{{\mathscr{A}}}  
\newcommand{\sB}{{\mathscr{B}}}  
\newcommand{\cC}{{C}}   
 
\newcommand{\cI}{{\mathcal{I}}}  

\newcommand{\sJ}{{\mathscr{J}}}  

\newcommand{\sK}{{\mathscr{K}}}  %
\newcommand{\Lp}{{L}}            
\newcommand{\Lpl}{L_{\text{loc}}}            

\newcommand{\RR}{\mathds{R}}
\newcommand{\NN}{\mathds{N}}

\newcommand{\Rd}{{\mathds{R}^{d}}}
\DeclareMathOperator{\Exp}{\mathbb{E}}
\DeclareMathOperator{\Prob}{\mathbb{P}}
\newcommand{\D}{\mathrm{d}}
\newcommand{\E}{\mathrm{e}}
\newcommand{\Ind}{\mathds{1}}   

\newcommand{\Sob}{{\mathscr W}}    
\newcommand{\Sobl}{{\mathscr W}_{\mathrm{loc}}} 

\newcommand{\df}{\coloneqq}

\DeclareMathOperator*{\trace}{Tr}

\DeclareMathOperator*{\supp}{support}

\newcommand{\grad}{\nabla}

\newcommand{\abs}[1]{\lvert#1\rvert}
\newcommand{\norm}[1]{\lVert#1\rVert}
\newcommand{\babs}[1]{\bigl\lvert#1\bigr\rvert}

\usepackage{color}
\definecolor{dmagenta}{rgb}{.4,.1,.5}
\definecolor{dblue}{rgb}{.0,.0,.5}
\definecolor{mblue}{rgb}{.0,.0,.7}
\definecolor{ddblue}{rgb}{.0,.0,.4}
\definecolor{dred}{rgb}{.7,.0,.0}
\definecolor{dgreen}{rgb}{.0,.5,.0}
\definecolor{Eeom}{rgb}{.0,.0,.5}

\newcommand{\ttl}{\Large Generalized principal eigenvalues on $\Rd$
of second order\\[5pt] elliptic operators with
rough nonlocal kernels }
\begin{document}
\title[Generalized eigenvalues with rough nonlocal kernels]
{\ttl}

\author[Ari Arapostathis]{Ari Arapostathis$^{\dag}$}
\address{$^{\dag}$Department of ECE,
The University of Texas at Austin,
EER~7.824, Austin, TX~~78712}
\email{ari@utexas.edu}

\author[Anup Biswas]{Anup Biswas$^\ddag$}
\address{$^\ddag$Department of Mathematics,
Indian Institute of Science Education and Research,
Dr.\ Homi Bhabha Road, Pune 411008, India}
\email{anup@iiserpune.ac.in}

\author[Prasun Roychowdhury]{Prasun Roychowdhury$^\ddag$}
\email{prasun.roychowdhury@students.iiserpune.ac.in}

\begin{abstract}
We study the generalized eigenvalue problem on the whole space for
a class of integro-differential elliptic operators.  The nonlocal operator
is over a finite measure, but this has no particular structure.
Some of our results even hold for
singular kernels.
The first part of the paper presents results concerning the existence of
a principal eigenfunction.
Then we present various necessary and/or sufficient conditions
for the maximum principle to hold, and use these to
characterize the simplicity of the principal eigenvalue.
\end{abstract}


\maketitle

\section{Introduction}

The analysis of eigenvalues and eigenfunctions is a central topic in the study of 
operator theory, partial differential equations and probability. The importance
of eigenvalue theory is evident from its wide range of applications including maximum
principles, bifurcation theory, stability analysis of nonlinear pde, large deviation
principle, risk-sensitive control etc. 
For important early work on the generalized principal eigenvalue of elliptic operators we refer the reader to the works of
Protter-Weinberger \cite{PW66}, Donsker-Varadhan \cite{DV76},
 Nussbaum \cite{N84}, Nussbaum-Pinchover \cite{NP92}.
In their seminal work 
Berestycki--Nirenbarg--Varadhan \cite{BNV-94} study the
properties of generalized Dirichlet principal eigenvalue of uniformly
elliptic operators in bounded domains and
show that the validity of maximum principle in bounded domains is equivalent to
the positivity of the  principal  eigenvalue.
This work has been extended for 
different kinds operators both in bounded and unbounded domains. See for instance,
Armstrong \cite{Arm-09}, Quaas-Sirakov \cite{QS08}, Ishii-Yoshimura \cite{IY-06},
Juutinen \cite{Ju-07} for bounded domains and 
Berestycki-Hamel-Rossi \cite{BHR-07}, Berestycki-Rossi \cite{Berestycki-15},
Biswas-Roychowdhury \cite{BR-20}, Nyguen-Vo \cite{NV-19} for unbounded domains.
In this article we are interested in the eigenvalue
theory of nonlocal operators
taking the form 
\begin{equation*}
\cI f(x)\,=\, \trace\bigl(a(x) \grad^2 f\bigr) + b(x)\cdot\grad f(x)
+ c(x) f(x) + I[f, x]\,,
\end{equation*}
where 
\begin{equation*}
I[f, x]\,=\,\int_{\Rd} \bigl(f(x+z)-f(x)\bigr)\,\nu(x,\D{z})\,,
\end{equation*}
and $\nu(x, \cdot)$ is a finite, non-negative Borel measure on $\Rd$.
It is easily seen that $\cI$ belongs to a large family of
integro-differential operators.
Very recently, generalized principal eigenvalues of integro-differential
operators have been studied in bounded domains, see for instance, 
Arapostathis-Biswas \cite{AB19}, Quaas-Salort-Xia \cite{QSX20},
Biswas \cite{Biswas-20}, Biswas-L\H{o}rinczi 
\cite{BL-19}, Pinsky \cite{Pinsky12,Pinsky09}, Dipierro-Proietti Lippi-Valdinoci \cite{DPV21} and references therein.
 To the best of our knowledge, there are only few works in nonlocal setting
dealing with the eigenvalue problems in unbounded domains. 
Berestycki-Roquejoffre-Rossi \cite{BRR-11} discuss eigenvalue problems 
for the fractional Laplacian in $\Rd$ for a periodic patch model. There are 
quite a few 
works on generalized eigenvalue problems in
unbounded domains with dispersal nonlocal kernels, see for example,
Berestycki-Coville-Vo \cite{BCV-16}, Coville-Hamel \cite{CH-20}, Coville \cite{Cov10}, Rawal-Shen \cite{RS12}, Shen-Xie \cite{SX15} and references therein.

The very first question that one would ask while studying the eigenvalue problem in
$\Rd$ is the existence of a principal eigenfunction.
In the case of nondegenerate local
elliptic pde (i.e., when $\nu=0$) this is obtained by passing to the limit in the
Dirichlet eigenvalue problems over an increasing sequence of balls covering $\Rd$,
with the help of Harnack inequality (cf. \cite{Berestycki-15}).
In the case of integro-differential equations, one also needs to
control the tail behaviour of the limiting eigenfunction to justify the passage
to the 
limit in the nonlocal integration (cf. \cite{AB19, BCV-16, CH-20}). These two key factors 
(i.e., Harnack inequality and tail behaviour of the eigenfunction) make the nonlocal eigenvalue problem in unbounded domains difficult (see also the discussions in
\cite[p~2711]{BCV-16} in the context of nonlocal dispersal kernels). For the particular operator 
$\cI$ above
it is known that the Harnack inequality does not hold in general 
\cite[Example~1.1]{ACGZ}. Therefore the existence of principal eigenfunctions becomes
non-obvious. In \cref{T1.1,T1.2} we identify a large family of operators for which
a
Harnack type inequality holds, and therefore, 
existence of a principal eigenfunction can
be proved \cref{T1.3}. We also show that if $\nu(x, \Rd)\to 0$ at infinity then
under some additional mild hypotheses we can assert the existence of
a principal eigenfunction.
In \cref{S2} we discuss the maximum principle for the operator $\cI$ in $\Rd$ and its
relation with generalized eigenvalues. These results are in the spirit of 
\cite{Berestycki-15}. Let us also mention \cref{T3.8}, where we prove an equivalence 
relation between minimal growth at infinity (\cref{GS}) and 
the strict monotonicity property
of principal eigenvalue on the right (\cref{D2.3}). In \cref{Appen}
we gather some known results
concerning Dirichlet eigenvalue problems in bounded domains.


\section{Existence of principal eigenfunction}
We deal with a second order, linear nonlocal operator given by
\begin{equation}\label{E1.1}
\cI f(x)\,=\, \trace\bigl(a(x) \grad^2 f\bigr) + b(x)\cdot\grad f(x)
+ c(x) f(x) + I[f, x]\,,
\end{equation}
where 
\begin{equation*}
I[f, x]\,=\,\int_{\Rd} \bigl(f(x+z)-f(x)\bigr)\,\nu(x,\D{z})\,.
\end{equation*}
Let us also denote by
\begin{equation}\label{E1.2}
\sA f(x) = \trace\bigl(a(x) \grad^2 f\bigr) + b(x)\cdot\grad f(x)
+ c(x) f(x) - \nu(x) f(x),
\end{equation}
where $\nu(x)=\nu(x,\Rd)$. It should be observed that the value $\nu(x,\{0\})$
has no effect on the equation,
and thus without any loss of generality we assume that $\nu(x,\{0\})=0$.
Throughout this article we impose the following assumptions, unless we state
otherwise.

\begin{assumption}\label{A1.1}
The coefficients satisfy the following conditions.
\begin{itemize}
\item[(i)] $b:\Rd\to\Rd$ and $c:\Rd\to\RR$ are locally bounded.
\item[(ii)] $a:\Rd\to\RR^{d\times d}$ is continuous and locally uniformly elliptic.
\item[(iii)] $\nu$ is a nonnegative, Borel measure satisfying the following:
\begin{itemize}
\item
The map $x\mapsto\nu(x,\Rd)\df\nu(x)$ locally bounded.
\item
$\nu$ has \emph{locally compact support} in the sense that
for every
compact set $K\subset\Rd$ there exists a compact set $K_1\subset\Rd$ so that
$\supp\bigl(\nu(x,\cdot\,)\bigr)\subset K_1$ for all $x\in K$.
\end{itemize} 
\end{itemize}
\end{assumption}
We are interested in the notion of generalized principal eigenvalue of $\cI$ in $\Rd$. 
Generalizing Berestycki-Rossi \cite{Berestycki-15}, which is actually in the 
spirit of \cite{BNV-94,NP92}, we define the principal
eigenvalue as follows
\begin{equation}\label{E1.3}
\lambda_1(\cI)\,=\,\sup\,\bigl\{\lambda\in\RR\,\colon\, \exists\; \text{positive~}
\phi\in\Sobl^{2,d}(\Rd) \text{~satisfying~} \cI \phi+\lambda \phi\le 0 \text{~in~}
\Rd \bigr\}.
\end{equation}
This definition can also been as a generalization of the Dirichlet principal
eigenvalue in bounded domains.
Let $D$ be a smooth bounded domain.
Then the Dirichlet principal eigenvalue of $\cI$ in $D$ is defined as 
\begin{equation}\label{E1.4}
\begin{aligned}
\lambda(\cI,D) &\,\df\,\sup\,\Bigl\{\lambda\in\RR\,\colon\, \exists\,
\phi\in\cC_+(\Rd)\cap\Sobl^{2,d}(D) \text{~satisfying~} \cI \phi+\lambda \phi\le 0
\text{~in~} D\\
&\mspace{500mu} \text{~and~} \phi>0 \text{~in~} D\Bigr\}\,,
\end{aligned}
\end{equation}
where $\cC_+(\Rd)$ denotes the subspace of $\cC(\Rd)$ consisting
of nonnegative functions.
The existence of principal eigenfunction in $D$ follows from \cref{TA1} in \cref{Appen}.
Also, note that $\lambda(\cI,D)$ is
decreasing with respect to increasing domains i.e. for $D_1\subset D_2$ we have
$\lambda(\cI,D_1)\ge \lambda(\cI,D_2)$.
It is also evident from this definition that $\lambda(\cI,D)\ge \lambda_1(\cI)$ for all
$D\subset \Rd$.

Now the following questions emerge in a natural manner.
\begin{itemize}
\item[\hypertarget{Q1}{\bf{Q1}.}]
Is $\lim_{n\to\infty}\lambda(\cI,B_n)$ equal to $\lambda_1(\cI)$?
\item[\hypertarget{Q2}{\bf{Q2}.}]
Does there exist a principal eigenfunction attaining the value $\lambda_1(\cI)$?
\end{itemize}
The above questions are quite related to each other. Recall that for 
nondegenerate second order elliptic operators (i.e., $\nu=0$)
the answers to the above questions are affirmative.
In fact, the equality in the question Q1 implies existence of a principal eigenfunction.
The main machinery in obtaining these results is the Harnack inequality
(cf. \cite{Berestycki-15}).
It is known from \cite[Example~1.1]{ACGZ} that the Harnack inequality does not hold
true for $\cI$, in general.
Also, a recent work of Mou \cite{M19} confirms only a weak-Harnack inequality for $\cI$.
Very recently, Harnack inequality has been studied for various integro-differential operators. For instance,
Foondun \cite{F09} obtains Harnack inequality for the bounded harmonic functions of second order integro-differential operators, 
Bass-Levin \cite{BL02}, Caffarelli-Silvestre \cite{CS09} consider Harnack estimate for the fractional Laplacian type
operators,
 Di Castro et. al. \cite{DKP14} establish similar estimate for the fractional $p$-Laplacian whereas
Coville \cite{Cov12}, Coville-Hamel \cite{CH-20} derive Harnack estimate for dispersal type nonlocal kernels.
One of the main contributions of this article is to
produce a large family of kernel $\nu$ for which a suitable Harnack type inequality holds,
and thus resolve Q1 and Q2 for this family of kernels.

In \hyperlink{H1}{\rm{(H1)}} and \hyperlink{H2}{\rm{(H2)}} which follow, we
describe two classes of kernels for which we can obtain a Harnack type estimate.

\begin{itemize}
\item[\hypertarget{H1}{\rm{(H1)}}]
The measure $\nu$ takes the form
$\nu(x,\D{y})= g(x, y) \D{y}$ for some measurable
function $g\colon\Rd\times\Rd\to[0, \infty)$.
There exists a nondecreasing function $\gamma\colon(0,\infty)\to(0,\infty)$
and positive functions $M_1, M_2$ defined on $(0,\infty)$, such that
for all $R>0$ the following hold:
\begin{subequations}
\begin{align}
g(x,y) &\,=\, 0\qquad \forall (x,y)\in B_R\times\Bar{B}^c_{\gamma(R)}\,,
\label{EH1-a}\\
g(x,y) &\,\le\, M_1(R) \qquad \forall (x,y)\in B_R\times\Bar{B}_{\gamma(R)}\,,
\label{EH1-b}\\
\int_{B_{\gamma(R)}} g(y, x-y)\, \D{y} &\,\ge\, M_2(R)
\qquad \forall x\in B_R\,.
\label{EH1-c}
\end{align}
\end{subequations}
\item[\hypertarget{H2}{\rm{(H2)}}]
For $d=1$, $\nu$ is locally compactly supported, that is, for some function $\gamma$ as in \hyperlink{H1}{\rm{(H1)}}
we have  $\supp\bigl(\nu(x,\cdot\,)\bigr)\subset B_{\gamma(R)}$ for all $x\in B_R$.
\end{itemize}

Here, and in the rest of the paper, we use the notation $B_r$ to denote the ball of radius $r$ centered at $0$.

%
%

\begin{remark}
Note that \cref{EH1-a} is equivalent to the statement that
$\nu$ has locally compact support (see \cref{A1.1}\,(iii)),
and \cref{EH1-b} implies that $\nu$ is locally bounded.
Hypothesis \cref{EH1-c} is a local positivity assumption.
It is satisfied, for example, if $g(x,y) \ge c \Ind_{B_\epsilon}(x-y)$ for some
positive constants $c$ and $\epsilon$ depending on $R$.
An example of a measure $\nu$ that does not satisfy \cref{EH1-c} is
given by
\begin{equation*}
g(x,y) \,=\, \Ind_{B_{2\abs{x}}\setminus B_{\abs{x}}}(x+y)\,.
\end{equation*}
This is because, if we evaluate \cref{EH1-c} at $x=0$, we obtain
\begin{equation*}
\int_{B_{\gamma(R)}} g(y, x-y)\, \D{y}\,=\,
\int_{B_{\gamma(R)}} \Ind_{B_{2\abs{y}}\setminus B_{\abs{y}}}(0)\, \D{y}
\,=\, 0\,.
\end{equation*}
On the other hand, if we let 
$$ g(x, y) = \Ind_{B_{\kappa\abs{x}}}(x+y),\quad \text{for some}\, \kappa>0,$$ 
then setting $\gamma(R)=(1+\kappa)R$, it can be easily checked that \hyperlink{H1}{\rm{(H1)}} holds.
\end{remark}

\begin{example}
If $\nu$ is translation invariant and has a density, that is,
$\nu(x,\D{y})=g(y)\D{y}$, then \cref{EH1-c} is always satisfied
unless $\nu\equiv0$.
Therefore, translation invariant measures $\nu$ which have a bounded
density with compact support satisfy \hyperlink{H1}{\rm{(H1)}}.
Then, in view of \cref{T1.3} which appears later in this section,
questions \hyperlink{Q1}{\rm{Q1}} and \hyperlink{Q2}{\rm{Q2}}
have an affirmative answer for this class of operators.
The reader should note that an even larger class of measures
satisfying \cref{EH1-c} are those that can be minorized
by a translation invariant measure with density.
For instance, a typical such example of measure $\nu$ satisfying \hyperlink{H1}{\rm{(H1)}} would be $\nu(x, \D{y})=\Ind_{B_{\xi(x)}(0)}(y) k(x, y) g(y)\D{y}$
where the function $g:\Rd\times\Rd\to[0, \infty)$ and 
$\xi:\Rd\to (0, \infty)$ are locally bounded, $k$ is bounded from above and 
below by positive constants on every compact sets,
 and for some $\delta>0$ we have
$$\int_{B_\delta} g(y)\D{y}>0.$$

\end{example}

Next we prove a Harnack type estimate.

\begin{theorem}\label{T1.1}
Let $a$ be locally Lipschitz,
and \hyperlink{H1}{\rm{(H1)}} or \hyperlink{H2}{\rm{(H2)}} hold.
Set $\Breve\gamma(R)\df\gamma\bigl(2R+\gamma(2R)\bigr)\vee 2R$.
Then, for every $R>0$ there exists a positive constant $C(R)$ such
that every nonnegative solution $u$ of $\cI u=0$ in
$B_{2\Breve\gamma(R)}$
satisfies
\begin{equation}\label{ET1.1A}
\sup_{B_R}\, u \,\le\, C(R) u(0)\,.
\end{equation}
\end{theorem}

\begin{proof}
Let
\begin{equation*}\sJ(x) \,\df\, \int_{\Rd} u(x+z) \nu(x,\D{z})\,,\end{equation*}
and recall the definition of $\sA$ in \cref{E1.2}.
Thus 
\begin{equation}\label{ET1.1B}
\sA u(x) \,=\, -\sJ(x)\quad \text{in~} B_{2\Breve\gamma(R)}
\end{equation}
by the hypothesis of the theorem.
Applying \cite[Theorem~9.20 and 9.22]{GilTru} to \cref{ET1.1B}, we obtain
\begin{equation}\label{ET1.1C}
\sup_{B_R}\, u \,\le\, C_0\biggl(\inf_{B_{R}}\, u+ \norm{\sJ}_{L^d(B_{2R})}\biggr)
\,\le\, 
C_0\left(u(0)+ \norm{\sJ}_{L^d(B_{2R})}\right)\,,
\end{equation}
for some positive constant $C_0$ depending only on $R$.
We continue by estimating $\norm{\sJ}_{L^d(B_{2R})}$.
Let $\{X_t\}_{t\ge0}$ be the diffusion process with generator
\begin{equation*}
\sA_\circ f \,=\, \trace\bigl(a(x) \grad^2 f\bigr) + b(x)\cdot\grad f(x)\,,
\end{equation*}
and $p^{B_{2\Breve\gamma(R)}}(t,x,y)$ denote the killed transition kernel
in $B_{2\Breve\gamma(R)}$, that is,
\begin{equation*}
\Prob_x(X_t\in A\, \colon t<\Hat\tau(R))
\,=\,\int_{A} p^{B_{2\Breve\gamma(R)}}(t,x,y)\, \D{y}\quad \text{for all Borel~} 
A\subset B_{2\Breve\gamma(R)}\,,
\end{equation*}
where $\Hat\tau(R)=\uptau_{B_{2\Breve\gamma(R)}}$ denotes the first exit time of $X$ from
$B_{2\Breve\gamma(R)}$.
Let
\begin{equation*}
\kappa_1(R)\,\df\, \sup_{x\in B_{2\Breve\gamma(R)}}\, \babs{c(x)-\nu(x)}\,.
\end{equation*}
Applying It\^{o}'s formula to \cref{ET1.1B}  we obtain
\begin{equation}\label{ET1.1D}
\begin{aligned}
u(0) &\,=\, \Exp_0\Bigl[u\bigl(X_{2\wedge\Hat\tau(R)}\bigr)\Bigr]
+ \Exp_0\left[\int_0^{2\wedge\Hat\tau(R)}
\E^{\int_0^t (c(X_s)-\nu(X_s))\D{s}} \sJ(X_t)\, \D{t}\right]\\
&\,\ge\, \E^{-2\kappa_1(R)} \Exp_0\left[\int_0^{2}
\Ind_{\{t<\Hat\tau(R)\}} \sJ(X_t)\, \D{t}\right].
\end{aligned}
\end{equation}
Since $p^{B_{2\Breve\gamma(R)}}(t,x,y)$ solves a parabolic equation,
it is known from the parabolic Harnack's inequality (see \cites{R05,Z96}) that
there exists a positive constant $\kappa_2(R)$ such that
\begin{equation*}
\inf_{t\in[1, 2]}\,\inf_{y\in B_{\Breve\gamma(R)}}\,p^{B_{2\Breve\gamma(R)}}(t,0,y)
\,\ge\, \kappa_2(R)\,.
\end{equation*}
Therefore, we have
\begin{equation}\label{ET1.1E}
\begin{aligned}
\Exp_0\left[\int_0^{2}\Ind_{\{t<\Hat\tau(R)\}}\,
 \sJ(X_t)\, \D{t}\right]
&\,\ge\,\int_1^2 \D{t} \int_{B_{2\Breve\gamma(R)}} \sJ(z)
p^{B_{2\Breve\gamma(R)}}(t, 0, z)\,\D{z}\\
&\,\ge\, \kappa_2(R) \int_{B_{\Breve\gamma(R)}} \sJ(z)\,\D{z}\,.
\end{aligned}
\end{equation}
Combining \cref{ET1.1D,ET1.1E}, we obtain
\begin{equation}\label{ET1.1F}
\int_{B_{\Breve\gamma(R)}} \sJ(z)\,\D{z}
\,\le\, \Bigl(\kappa_2(R)\E^{-2\kappa_1(R)}\Bigr)^{-1} u(0)\,.
\end{equation}

First assume 
\hyperlink{H1}{\rm{(H1)}}.
Then, using \cref{ET1.1F}, we have
\begin{equation}\label{ET1.1G}
\begin{aligned}
\Bigl(\kappa_2(R)\E^{-2\kappa_1(R)}\Bigr)^{-1} u(0)
&\,\ge\, \int_{B_{\Breve\gamma(R)}} \sJ(z)\,\D{z} \\
&\,=\, \int_{B_{\Breve\gamma(R)}}\biggl(\int_{\Rd} u(y) g(z, y-z)\,\D{y}\biggr)\, \D{z}\\
&\,\ge\,\int_{B_{\Breve\gamma(R)}}\biggl(\int_{B_{2R+\gamma(2R)}}u(y)
g(z,y-z)\,\D{y}\biggr)\,
\D{z}\\
&\,\ge\, M_2(2R+\gamma(2R)) \int_{B_{2R+\gamma(2R)}} u(y)\,\D{y}\,,
\end{aligned}
\end{equation}
where in the third inequality we use Fubini's theorem
and \cref{EH1-c} for the last inequality.
Thus, by \cref{EH1-a,EH1-b,ET1.1G}, we obtain for $x\in B_{2R}$ that
\begin{equation}\label{ET1.1H}
\begin{aligned}
\sJ(x) &\,\le\, M_1(R) \int_{x+B_{\gamma(2R)}} u(y)\,\D{y}\\
&\,\le\, M_1(R) \int_{B_{2R + \gamma(2R)}} u(y)\,\D{y}\\
&\,\le\, \frac{M_1(R)}{M_2(2R+\gamma(2R))\kappa_2(R)}\,\E^{2\kappa_1(R)}\,
u(0)\qquad \forall x\in B_{2R}\,.
\end{aligned}
\end{equation}
It is clear then that \cref{ET1.1A} follows from \cref{ET1.1C,ET1.1H}.

Under \hyperlink{H2}{\rm{(H2)}}, since $d=1$, \cref{ET1.1A} follows
from \cref{ET1.1F}.
This completes the proof.
\end{proof}

\begin{remark}
Under \hyperlink{H1}{\rm{(H1)}} or \hyperlink{H2}{\rm{(H2)}}
a slightly more general estimate holds.
Suppose that a nonnegative $u$ satisfies $\cI u =- f$ in
$B_{2\Tilde\gamma(R)}$,
for some nonnegative $f\in\Lpl^d(\Rd)$.
Then there exists a positive constant $C_R$ such that
\begin{equation*}
\sup_{B_R}\, u \,\le\, C_R\biggl(\inf_{B_{R}}\, u+ \norm{f}_{L^d(B_{2R})}\biggr)\,.
\end{equation*}
\end{remark}

We continue with the following definition.

\begin{definition}\label{D1.1}
We say that $\nu$ points inwards in a bounded domain $D$ if there exists
a domain $D'\Supset D$ such that
$\supp\bigl(\nu(x,\cdot\,)\bigr)\subset D-x$ for all $x\in D'$.
\end{definition}

\begin{theorem}\label{T1.2}
Let $a$ be locally Lipschitz, and $D$ a domain on which $\nu$ points inwards.
Then, for any bounded domain
$\widetilde{D}\Supset D$,
there exists a constant $C_{\mathsf{H}}$ such that every
nonnegative solution $u$ of $\cI u=0$ in $\widetilde{D}$
 satisfies
\begin{equation}\label{ET1.2A}
 u(x) \,\le\, C_{\mathsf{H}}\, u(y)\qquad \forall\,x,y\in D\,.
\end{equation}
\end{theorem}

\begin{proof}
Let $D'$ be as in \cref{D1.1}.
As in \cref{ET1.1D}, we have
\begin{equation*}
\sup_{D}\, u \,\le\, C_0\biggl(\inf_{D}\, u+ \norm{\sJ}_{L^d(D'\cap\widetilde{D})}\biggr)
\end{equation*}
for some positive constant $C_0$.
Thus, either $\sup_{D}\, u \,\le\, 2 C_0\inf_{D} u$, in which case
\cref{ET1.2A} holds with $C_{\mathsf{H}}=2C_0$, or
\begin{equation}\label{ET1.2B}
\sup_{D}\, u \,\le\, 2 C_0\, \norm{\sJ}_{L^d(D'\cap\widetilde{D})}\,.
\end{equation}
Since $\nu$ points inwards in $D$, \cref{ET1.2B} implies that
\begin{equation}\label{ET1.2C}
\sup_{D'\cap\widetilde{D}}\, \sJ \,\le\, C_0'\, \norm{\sJ}_{L^d(D'\cap\widetilde{D})}
\end{equation}
for some positive constant $C_0'$.
From \cref{ET1.2C}, using the Minkowski inequality, we obtain
\begin{equation*}
\sup_{D'\cap\widetilde{D}}\, \sJ \,\le\, C_0'\, \norm{\sJ}_{L^d(D'\cap\widetilde{D})}
\,\le\, C_0'\,\biggl(\sup_{D'\cap\widetilde{D}} \sJ\biggr)^{\frac{d-1}{d}}\,
\norm{\sJ}^{\nicefrac{1}{d}}_{L^1(D'\cap\widetilde{D})}\,,
\end{equation*}
which implies that
\begin{equation}\label{ET1.2D}
\sup_{D'\cap\widetilde{D}}\, \sJ \,\le\, C_0'\, \norm{\sJ}_{L^1(D'\cap\widetilde{D})}\,.
\end{equation}
\Cref{ET1.1B,ET1.2D} imply a Harnack estimate of the
form \cref{ET1.2A} by \cite[Corollary~2.2]{AA-Harnack}.
\end{proof}

Consider the following hypothesis.

\begin{itemize}
\item[\hypertarget{H3}{\rm{(H3)}}]
There exists an increasing sequence of bounded domains $\{D_n\}_{n\in\NN}$
which covers $\Rd$ and $\nu$ points inwards on $D_n$ for all $n\in\NN$.
\end{itemize}

Recall the definitions in \cref{E1.3,E1.4}.
We have the following theorem concerning
 the existence of a principal eigenfunction for $\lambda_1(\cI)$ on $\Rd$.

\begin{theorem}\label{T1.3}
Grant any of the hypotheses \hyperlink{H1}{\rm{(H1)}}--\hyperlink{H3}{\rm{(H3)}},
Then, the following hold:
\begin{itemize}
\item[(a)]
We have $\lim_{R\to\infty}\lambda(\cI,B_R)=\lambda_1(\cI)$
under \hyperlink{H1}{\rm{(H1)}} or \hyperlink{H2}{\rm{(H2)}},
and 
$$\lim_{n\to\infty}\lambda(\cI,D_n)=
\lim_{n\to\infty}\lambda(\cI,B_n)=\lambda_1(\cI)
$$
under \hyperlink{H3}{\rm{(H3)}}.
\item[(b)]
Suppose that $\lambda_1(\cI)>-\infty$.
Then,
for any $\lambda\le \lambda_1(\cI)$, there exists a positive
$\varphi\in\Sobl^{2,p}(\Rd)$, $p>d$, satisfying
\begin{equation*}
\cI \varphi+\lambda \varphi\,=\, 0\quad \text{in~} \Rd\,.
\end{equation*}
\end{itemize}
\end{theorem}

\begin{proof}
We start with part (a).
Since we have a Harnack inequality under the above hypotheses, the proof
follows a standard argument (cf.\ \cite{Berestycki-15}).
We show that $\lim_{n\to\infty}\lambda(\cI,B_n)=\lambda_1(\cI)$,
and there exists a positive
eigenfunction corresponding to the eigenvalue $\lambda_1(\cI)$.
Without loss of generality, we assume that $D_n=B_n$.
Define $\Hat\lambda\df\lim_{n\to\infty}\lambda(\cI,B_n)$.
It follows from this definition that $\Hat\lambda\ge \lambda_1(\cI)$.
If  $\Hat\lambda=-\infty$, then $\lambda_1(\cI)=-\infty$ and there is nothing to prove.
So we suppose that $\Hat\lambda\in (-\infty, \infty)$.
Let $\psi_n$ be the Dirichlet principal eigenfunction in 
$B_n$ corresponding to the eigenvalue $\lambda(\cI,B_n)$ (see \cref{TA1}) i.e. 
\begin{equation}\label{ET1.3A}
\begin{aligned}
\cI \psi_n &\,=\, -\lambda(\cI,B_n)\, \psi_n\quad \text{in~} B_n\,,\\
\psi_n&\,=\,0\quad \text{in~} B_n^c\,,\\
\psi_n&>0\quad \text{in~} B_n\,.
\end{aligned}
\end{equation}
Applying \cref{T1.1} or \cref{T1.2} we see that for any compact set $K$,
there exists an integer $m(K)$ such that 
\begin{equation}\label{ET1.3B}
\sup_{n\ge m(K)}\, \sup_{K}\, \psi_n\,\le\, C\,,
\end{equation}
for some constant $C$ dependent on $K$. Writing \cref{ET1.3A} as
\begin{equation*}
\sA\psi_n=-\sJ_n(x) \,=\, -\int_{\Rd} \psi_n(x+y) \nu(x,\D{y})\,.
\end{equation*}
Then, by \cref{A1.1}\,(iii) and \cref{ET1.3B}, we obtain
\begin{equation*}
\sup_{n\ge m_1(K)}\, \sup_{K}\, \sJ_n\,\le\, C_1
\end{equation*}
for some constant $C_1$ and some integer $m_1(K)\ge m(K)$.
Thus, from standard elliptic estimates, we obtain a constant $C_2$ for every
$p>d$ satisfying
\begin{equation*}
\sup_{n\ge m_2(K)}\,  \norm{\psi_n}_{\Sob^{2,p}(K)} \,\le\,C_2
\end{equation*}
for some integer $m_2(K)$. Therefore, using a standard diagonalization argument
we can extract a 
subsequence $\psi_{n_k}$ converging to a nonnegative $\psi\in\Sobl^{2,p}(\Rd)$, $p>d$,
and $\psi(x_0)=1$
as $n_k\to\infty$. By \cref{A1.1} we also obtain
\begin{equation*}\cI\psi + \Hat\lambda\psi\,=\, 0\quad \text{in~} \Rd.\end{equation*}
Since $(\sA +\Hat\lambda)\psi\le 0$, by the strong maximum principle
we also have $\psi>0$ in $\Rd$.
Therefore, $\psi$ is an admissible test function in \cref{E1.3} implying
that $\Hat\lambda\le\lambda_1(\cI)$. Hence
$\Hat\lambda=\lambda_1(\cI)$ and this completes the proof of step 1.

We continue with part (b) and show that for every $\lambda<\lambda_1(\cI)$ there exists a
positive eigenfunction.
Fix $\lambda<\lambda_1(\cI)$. Let $f_n$ be a smooth, non-positive, nonzero function with 
$\supp(f_n)\Subset B_n\setminus B_{n-1}$.
By \cref{TA3}, since $\lambda(\cI,B_n)-\lambda>0$,
 there exists a unique positive $\varphi_n$ satisfying
\begin{equation*}
\begin{split}
\cI \varphi_n + \lambda\varphi_n &\,=\, f_n \quad \text{in~} B_n\,,\\
\varphi_n &\,=\,0 \quad \text{in~} B_n^c\,.
\end{split}
\end{equation*}
Normalize $\varphi_n$ to satisfy $\varphi_n(x_0)=1$ and then we can use the
Harnack inequality as in Step~1 to extract a converging subsequence of
$\varphi_n$ to some positive function
$\varphi\in\Sobl^{2,p}(\Rd)$, $p>d$.
Thus, we obtain $\cI \varphi + \lambda\varphi=0$ in $\Rd$. This completes the proof.
\end{proof}

Our next result deals with another class of kernels not covered by \cref{T1.3}.
We denote by $\lambda(\sA,D)$ the Dirichlet principal eigenvalue of $\sA$
in a smooth domain $D$, possibly unbounded, i.e.,
\begin{equation*}
\lambda(\sA,D)\,\df\,
\sup\,\Bigl\{\lambda\in\RR\,\colon \exists \, \phi\in\cC_+(\Rd)\cap\Sobl^{2,d}(D)
\text{~satisfying~} \sA \phi+\lambda \phi \,\le\, 0 \text{~in~} D \text{~and~}\phi>0
\text{~in~}D\Bigr\}\,.
\end{equation*}
For $D=\Rd$ this principal eigenvalue will be denoted as $\lambda_1(\sA)$.
It is evident from the definition that $\lambda(\sA,D)\ge \lambda(\cI,D)$.

\begin{theorem}\label{T1.4}
We assume the following:
\begin{itemize}
\item[(1)]
The matrix $a$ is bounded, uniformly elliptic and uniformly continuous in $\Rd$.
\item[(2)]
For
some $\alpha\ge 0$ we have $|b(x)|\le C(1+\abs{x}^\alpha)$ and
$|c(x)|\le C(1+\abs{x}^{2\alpha})$
for all $x\in\Rd$.
\item[(3)]
There exists some open ball $\sB_\circ\subset\Rd$ such that
$\supp\bigl(\nu(x,\cdot\,)\bigr)\subset\sB_\circ$ for all $x\in\Rd$, and
\begin{equation}\label{ET1.4A}
\lim_{\abs{x}\to \infty}\nu(x,\Rd)\,\E^{\abs{x}^\alpha} \,=\, 0\,.
\end{equation}
\end{itemize}
Then
\begin{equation}\label{ET1.4B}
\lim_{n\to\infty}\,\lambda(\cI,B_n)\,=\,\lambda_1(\cI)\,.
\end{equation}
In addition, if we also have
\begin{equation}\label{ET1.4C}
\lim_{n\to\infty} \lambda(\sA,B_n^c)\,>\,\lambda_1(\cI)\,>\,-\infty\,,
\end{equation} 
then there exists a principal eigenfunction for $\lambda_1(\cI)$.
\end{theorem}

\begin{proof}
We first show that $\lambda_1(\sA)>-\infty$ and
$\sup_{x\in\Rd}\,\nu(x,\Rd)\,\E^{\abs{x}^\alpha}<\infty$ imply that
$\lambda_1(\cI)>-\infty$.
This is because if $\lambda_1(\sA)>-\infty$, then
there exists a positive $\varphi\in \Sobl^{2,p}(\Rd)$, $p>d$, satisfying
\begin{equation}\label{PT1.4A}
\sA \varphi + \lambda_1(\sA) \varphi
\,=\, 0\quad \text{in~} \Rd\,.
\end{equation}
By the assumption on the coefficients of the operator $\sA$,
and using \cite[Lemma~4.1]{ABBK-19}, we obtain from \cref{PT1.4A} that
\begin{equation}\label{PT1.4B}
\frac{\abs{\grad \varphi(x)}}{\varphi(x)}
\,\le\, C_1(1+|x|^\alpha)\quad \forall\,x\in\Rd,
\end{equation}
for some constant $C_1$.
Let $f(x)\df\log \varphi(x)$.
Then,  it follows from \cref{PT1.4B} that $\abs{\grad f(x)}\le C_1(1+|x|^\alpha)$.
Hence we have
\begin{equation*}
\sup_{y\in\sB_\circ}\,\abs{f(x+y)-f(x)} \,\le\, C_2(1+|x|^\alpha)\quad
\forall\, x\in\Rd\,.
\end{equation*}
This implies that
\begin{equation*}
\sup_{x\in\Rd}\, \sup_{y\in\sB_\circ}\,\frac{\varphi(x+y)}{\varphi(x)}
\,<\, C_3 \E^{\abs{x}^\alpha}
\end{equation*}
for some constant $C_3$.
Thus we obtain
\begin{equation}\label{PT1.4C}
\begin{aligned}
\frac{1}{\varphi} \cI \varphi &\,=\, \frac{1}{\varphi} \sA \varphi
+ \int_{\sB_\circ} \frac{\varphi(x+y)}{\varphi(x)} \nu(x,\D y)\\
&\,\le\,  -\lambda_1(\sA) + 
C_3 \sup_{x\in\Rd}\,\nu(x,\sB_\circ) \E^{|x|^\alpha}\,,
\end{aligned}
\end{equation}
which shows that $\lambda_1(\cI)>-\infty$.

We proceed to prove \cref{ET1.4B}.
In view of the preceding paragraph, if $\lambda_1(\cI)=-\infty$,
then we must have
\begin{equation*}
-\infty \,=\, \lambda_1(\sA) \,=\,
\lim_{n\to\infty} \lambda(\sA,B_n) \,\ge\,
\lim_{n\to\infty} \lambda(\cI,B_n) \,\ge\, \lambda_1(\cI)
\end{equation*}
and \cref{ET1.4B} holds.
Therefore,
without loss of generality, we
assume that $\lambda_1(\cI)>-\infty$.
For $\delta\in\RR$ we define 
\begin{equation*}
\widetilde\lambda(\delta) \,=\,
\lim_{n\to\infty} \lambda(\cI + \delta \Ind_{B_1},B_n).
\end{equation*}
From this definition it follows that 
$\widetilde\lambda(\delta)\ge \lambda_1(\cI+\delta \Ind_{B_1})$
for all $\delta\in\RR$.
We prove \cref{ET1.4B} using the argument of contradiction.
Suppose that $\widetilde\lambda(0)> \lambda_1(\cI)$.
Since $\delta\mapsto \lambda(\cI + \delta \Ind_{B_1},B_n)$ is concave and
decreasing \cite[Theorem~2.3]{AB19}, it follows that
$\delta\mapsto \widetilde\lambda(\delta)$ is concave and thus continuous.
Moreover,
\begin{equation*}
\widetilde\lambda(\delta) \,\le\, \lambda(\cI + \delta \Ind_{B_1},B_1)
\,=\, \lambda(\cI,B_1)-\delta\rightarrow\,-\infty
\end{equation*}
as $\delta\to\infty$.
This means that we can select $\delta>0$ such that
\begin{equation}\label{ET1.4D}
\widetilde\lambda(0) \,>\,\widetilde\lambda(\delta) \,>\, \lambda_1(\cI)\,.
\end{equation}
Since
\begin{equation*}
\lambda_1(\sA) \,=\,
\lim_{n\to\infty} \lambda(\sA,B_n) \,\ge\,
\lim_{n\to\infty} \lambda(\cI,B_n)\,=\,\widetilde\lambda(0)\,,
\end{equation*}
 it also holds that
$\lambda_1(\sA)> \widetilde\lambda(\delta)$.
Let
$\varepsilon\df \frac{1}{2}\bigl(\lambda_1(\sA)-\widetilde\lambda(\delta)\bigr)>0$.

We write \cref{PT1.4A} as
\begin{equation}\label{ET1.4E}
\sA \varphi + \bigl(\widetilde\lambda(\delta)+2\varepsilon\bigr) \varphi
\,=\, 0\quad \text{in~} \Rd\,.
\end{equation}
Thus, using \cref{ET1.4A,PT1.4C,ET1.4E}, we obtain
\begin{equation}\label{ET1.4G}
\begin{aligned}
\frac{1}{\varphi} \cI \varphi &\,=\, \frac{1}{\varphi} \sA \varphi
+ \int_{\sB_\circ} \frac{\varphi(x+y)}{\varphi(x)} \nu(x,\D y)\\
&\,\le\,  -\widetilde\lambda(\delta)-2\varepsilon + 
C_3 \nu(x,\sB_\circ) \E^{|x|^\alpha}\\
&\,<\, -\widetilde\lambda(\delta)-\varepsilon
\end{aligned}
\end{equation}
for all $x$ outside a ball $B_{k_0}$, for some $k_0\in\NN$.
Now consider the Dirichlet eigenvalue problem
\begin{equation}\label{ET1.4H}
\cI \psi_n + \delta \Ind_{B_1}\psi_n + \lambda(\cI + \delta \Ind_{B_1},B_n) \psi_n
\,=\, 0 \quad \text{in~} B_n\,,
\quad \text{and}\quad \psi_n=0\quad \text{on~} B_n^c\,.
\end{equation}
Let
\begin{equation*}
\kappa_n\,\df\,\min_{B_n}\frac{\varphi}{\psi_n}\,,\quad\text{and}\quad
\varphi_n\,\df\,\kappa_n\psi_n\,.
\end{equation*}
Fix $n_0$ large so that $\lambda(\cI+\delta \Ind_{B_1}, B_n)<\tilde{\lambda}(\delta)
+\varepsilon$ for all $n\geq n_0$.
By \cref{ET1.4G,ET1.4H}, we have
\begin{equation}\label{ET1.4I}
\cI (\varphi-\varphi_n)
+ \lambda(\cI + \delta \Ind_{B_1},B_n) (\varphi-\varphi_n) \,<\,0
\qquad\forall\,x\in B_n\setminus B_{k_0}\,.
\end{equation}
Thus, since $\varphi\ge\varphi_n$ on $B_n$, it follows from
\cref{ET1.4I} and the
strong maximum principle that $\varphi-\kappa_n\psi_n$ can not vanish in
$B_n\setminus B_{k_0}$, which implies that the minimum
of $\frac{\varphi}{\psi_n}$
is attained in $B_{k_0}$ for all $n>k_0\vee n_0$.
We write \cref{ET1.4H} as
\begin{equation}\label{ET1.4J}
\sA\varphi_n + \delta \Ind_{B_1}\varphi_n + \lambda(\cI+\delta\Ind_{B_1},B_n) \varphi_n
\,=\, -\sJ_n \quad\text{in\ } B_n\,,
\end{equation}
with
\begin{equation*}
\sJ_n(x) \,\df\, \int_{\Rd} \varphi_n(x+z) \nu(x,\D{z}) \,\le\,
\int_{\Rd}\varphi(x+z)\nu(x,\D{z}) \,=:\,\widehat\sJ(x)\,.
\end{equation*}
By \cite[Theorem~9.20 and 9.22]{GilTru}, we obtain from \cref{ET1.4J} the
bound
\begin{equation*}
\sup_{B_R}\,\varphi_n\,\le\, C_R \biggl(\inf_{B_{2R}}\,\varphi_n
+ \norm{\sJ_n}_{L^d(B_{2R})}\biggr)
\,\le\, C_R \biggl(\sup_{B_{k_0}}\,\varphi
+ \norm{\widehat\sJ}_{L^d(B_{2R})}\biggr)\,,
\end{equation*}
which is valid for all $R$ and $n$ such that $n>2R>2k_0$.
Thus, by a standard elliptic estimate, the sequence
$\{\varphi_n\}_{n\in\NN}$ is uniformly bounded in $\Sob^{2,p}(B_R)$, $p>d$,
for any $R>0$, and this permits us to
extract a subsequence converging weakly in $\Sobl^{2,p}(\Rd)$ to some nonnegative
function $\psi\in\Sobl^{2,p}(\Rd)$.
Passing to the limit as $n\to\infty$ in \cref{ET1.4J}, we obtain
\begin{equation}\label{ET1.4K}
\sA\psi + \delta \Ind_{B_1}\psi + \widetilde\lambda(\delta) \psi 
\,=\, -\int_{\Rd}\psi(x+y)\nu(x,\D{y})\,.
\end{equation}
Since, as we showed earlier, $\varphi_n$ touches $\varphi$ at one point
from below in $\Bar{B}_{k_0}$,  we must have $\psi>0$ in $\Rd$
by the strong maximum principle.
\cref{ET1.4K} then implies that
$\cI\psi + \widetilde\lambda(\delta) \psi \le 0$ in $\Rd$,
which contradicts the inequality
$\lambda_1(\cI)< \widetilde\lambda(\delta)$ in \cref{ET1.4D}.
This completes the proof of \cref{ET1.4B}.

Next we establish the existence of a principal eigenfunction
under \cref{ET1.4C}.
Choose $n_\circ$ large enough so that
$\lambda(\sA,B_{n_\circ}^c)>\lambda_1(\cI)$,
and let 
$\varphi_{n_\circ}\in\Sobl^{2,p}(B^c_{n_\circ})$ be such that
\begin{equation}\label{ET1.4L}
\sA\varphi_{n_\circ} + \lambda(\sA,B_{n_\circ}^c) \varphi_{n_\circ}
\,=\, 0\quad \text{in~} B^c_{n_\circ},
\quad \text{and}\quad \varphi_{n_\circ}>0
\text{~in~} \Bar{B}^c_{n_\circ}\,.
\end{equation}
Existence of the eigenfunction $\varphi_{n_\circ}$ follows from \cite{Berestycki-15}.
Let $\zeta\colon\Rd\to [0,1]$ be a smooth cut-off function such
that $\zeta=1$ in $B_{n_\circ+1}$ and $\zeta=0$ in $B^c_{n_\circ+2}$.
Define $\Phi\df\zeta+ (1-\zeta)\varphi_{n_\circ}\in\Sobl^{2,p}(\Rd)$.
By the definition of $\zeta$ and \cref{ET1.4L},
we have $\sA \Phi + \lambda(\sA,B_{n_\circ}^c)\Phi=0$ in $\Bar{B}^c_{n_\circ+2}$.
This means that
\cref{PT1.4B} holds for all $x$ outside some compact set,
and we can follow the  argument in the first part of the proof (see \cref{ET1.4G}),
combined with the fact that $\supp\bigl(\nu(x,\cdot\,)\bigr)\subset\sB_\circ$,
to find a ball $B_{k_0}$ satisfying 
\begin{equation*}
\cI\Phi + (\lambda(\sA,B_{n_\circ}^c)-\varepsilon)\Phi \,<\, 0\quad \text{in~}
B_{k_0}^c\,,
\end{equation*}
where $\varepsilon>0$ is chosen
small enough to satisfy $\lambda(\sA,B_{n_\circ}^c)-\varepsilon>\lambda_1(\cI)$.
Now consider the principal Dirichlet eigenfunction in $B_n$ satisfying
\begin{equation}\label{ET1.4M}
\begin{split}
\cI \psi_n &\,=\, -\lambda(\cI,B_n)\, \psi_n\quad \text{in~} B_n\,,\\
\psi_n&\,=\,0\quad \text{in~} B_n^c\,,\\
\psi_n&>0\quad \text{in~} B_n\,.
\end{split}
\end{equation}
Let $\rho_n=\min_{B_n}\frac{\Phi}{\psi_n}$ and define $\widehat\psi_n=\rho_n\psi_n$.
Employing the argument in the first part of the proof, we can show
that $\widehat\psi_n$ touches $\Phi$ at one point from below in $\Bar{B}_{k_0}$,
and that the sequence $\{\widehat\psi_n\}_{n\in\NN}$ is bounded in
$\Sob^{2,p}(B_R)$, $p>d$, for any $R>0$.
We extract a subsequence converging in $\Sobl^{2,p}(\Rd)$, $p>d$, to some 
positive $\widehat\psi$, and pass to the limit in \cref{ET1.4M},
after multiplying this equation with $\rho_n$,
to obtain
$\cI \widehat\psi + \lambda_1(\cI)\widehat\psi=0$ in $\Rd$. This completes the proof.
\end{proof}

\begin{remark}\label{R1.4}
Note that \cref{ET1.4C} is used to construct a suitable super-solution that allows
us to pass the limit.
This can be replaced by the following condition:
there exists a positive $V\in\Sobl^{2,d}(\Rd)$ satisfying 
\begin{equation*}
\cI V + (\lambda_1(\cI) + \varepsilon) V  \,\le\, 0\quad \text{in~} K^c\,,
\end{equation*}
for some compact set $K$ and $\varepsilon>0$.
For instance, if $\lim_{|x|\to\infty} c(x)=-\infty$, then we can take $V=1$.
\end{remark}

\begin{remark}
If there exists a principal eigenfunction for $\lambda_1(\cI)$,
and $\cA$ is strictly right-monotone at $\lambda_1(\sA)$
(see \cite{ABS-19} and \cref{D2.3}) and $\nu$ is nontrivial, then
$\lambda_1(\cI)<\lambda_1(\sA)$, and therefore \cref{ET1.4C}  holds.
\end{remark}

\section{Maximum principles and simplicity of eigenvalues}\label{S2}

Throughout this section we assume that $\lambda_1(\cI)>-\infty$.
Following \cite{Berestycki-15}, we define the eigenvalues
$\lambda^\prime_1(\cI)$ and $\lambda^{\prime\prime}_1(\cI)$ by
\begin{align*}
\lambda^\prime_1(\cI) &\,\df\, \inf\,\Bigl\{\lambda\in\RR\,\colon
\exists \text{~positive~} \phi\in\Sobl^{2,d}(\Rd)\cap \Lp^\infty(\Rd)
  \text{~satisfying~} \cI \phi+\lambda \phi\ge 0 \text{~in~} \Rd \Bigr\}\,,
\intertext{and}
\lambda^{\prime\prime}_1(\cI) &\,\df\, \sup\,\Bigl\{\lambda\in\RR\,\colon
\exists\,  \phi\in\Sobl^{2,d}(\Rd)  \text{~satisfying~} \cI \phi+\lambda \phi \,\le\, 0
\text{~in~}\Rd \text{~and~} \inf_{\Rd}\phi>0 \Bigr\}\,.
\end{align*}

First, we develop a maximum principle.

\begin{definition}
We say that the operator $\cI$ satisfies the \emph{maximum principle}
(on $\Rd$) if for every function $u\in\Sobl^{2,d}(\Rd)$ 
satisfying $\cI u\ge 0$ in $\Rd$ and $\sup_{\Rd} u < \infty$, we have $u\le 0$ in $\Rd$.
\end{definition}

It is well-known  that the eigenvalues $\lambda^\prime_1(\cI)$ and
$\lambda^{\prime\prime}_1(\cI)$ are closely connected to
the maximum principle for the operator $\cI$ \cite{Berestycki-15,BR-20}.
We refer to the following two sets of conditions on the coefficients of $\cI$
in \cref{T2.1} below.
\begin{equation}\label{E_growth}
\sup_{\Rd}\, c(x)\,<\,\infty\,,\quad \sup_{\Rd}\, \abs{a_{ij}(x)}\,<\,\infty\,
\quad \text{and\ } 
\sup_\Rd\,\frac{b(x)\cdot x}{|x|}\,<\,\infty\,,
\end{equation}
and
\begin{equation}\label{P_growth}
\sup_{\Rd} \, c(x)\,<\,\infty\,,\quad \limsup_{\abs{x}\to\infty}
\,\frac{\abs{a_{ij}(x)}}{\abs{x}^2}\,<\,\infty\,\quad\text{and\ } 
\limsup_{\abs{x}\to\infty}\,\frac{b(x)\cdot x}{|x|^2}\,<\,\infty\,.
\end{equation}

Our next result provides a necessary and a sufficient condition for maximum principle.

\begin{theorem}\label{T2.1}
The following hold.
\begin{itemize}
\item[(i)] If $\cI$ satisfies the maximum principle then $\lambda^{\prime}_1(\cI)\ge 0$.
\item[(ii)] Suppose that  for some ball  $\sB_\circ$ we have 
$\supp\bigl(\nu(x,\cdot\,)\bigr)\subset\sB_\circ$ for all $x$, $\nu(x)$ is bounded,
and either \cref{E_growth} or \cref{P_growth} hold.
 Then $\cI$ satisfies the maximum principle if $\lambda^{\prime\prime}_1(\cI)>0$.
\item[(iii)] Suppose $a, b$ and $(c-\nu)^-$ are bounded.
Then $\cI$ satisfies the maximum principle if
$\lambda^{\prime\prime}_1(\cI)>0$.
\end{itemize}
\end{theorem}
\begin{proof}
(i)\, Let $\lambda^{\prime}_1(\cI)< 0$.
Then there exist a positive $\phi\in\Sobl^{2,d}(\Rd)\cap L^\infty(\Rd)$ satisfying
$\cI \phi \ge 0$ in $\Rd$. This clearly contradicts the maximum principle.
Thus if $\cI$ satisfies the maximum principle,
 it must be the case that $\lambda^{\prime}_1(\cI)\ge 0$.

(ii)\, Suppose, on the contrary, that $\lambda^{\prime\prime}_1(\cI)>0$,
 but  $\cI$ does not satisfy the maximum principle.
 Then there exists a $u\in\Sobl^{2,d}(\Rd)$ which is positive somewhere in $\Rd$
 and $\cI u \ge 0$ in $\Rd$ with $\sup_{\Rd} u < \infty$.
 Also, since $\lambda^{\prime\prime}_1(\cI)>0$, there exist $\lambda>0$
 and $\psi\in\Sobl^{2,d}(\Rd)$ satisfying  $\cI \psi+\lambda\psi\le 0$ in $\Rd$
 and $\inf_{\Rd} \psi >0$. 
By scaling appropriately we may also assume that $\psi\ge u$ in $\Rd$.
Now choose a smooth positive function $\chi:\Rd\rightarrow \mathbb{R}$
such that if $x\in B_{1}^c$, then
\begin{equation*}
\chi(x)\,=\,
\begin{cases}
\E^{\sigma\abs{x}} & \text{if~} \cI \text{~satisfies~} \cref{E_growth},
\\[5pt]
\abs{x}^\sigma & \text{if~} \cI \text{~satisfies~} \cref{P_growth}.
\end{cases}
\end{equation*}
After an easy computation we can write for $x\in B_1^c$ that
\begin{align}
\cI \chi(x)&\,\le\,
\biggl(\sigma \bigl(\sigma + \tfrac{d-1}{\abs{x}}\bigr) \norm{a(x)}
+ \sigma \frac{b(x)\cdot x}{|x|} + c(x)-\nu(x)+\int_{\sB_\circ}
\frac{\E^{\sigma\abs{x+z}}}{\E^{\sigma\abs{x}}}\nu(x,{\rm dz})\biggr)\chi(x)\,,
\label{PT2.1Aa}\\
\intertext{if $\cI$ satisfies \cref{E_growth}, and}
\cI \chi(x)&\,\le\,
\biggl((\sigma^2+d\sigma-2\sigma)\frac{\norm{a(x)}}{|x|^2}+\sigma\frac{b(x)\cdot x}
{|x|^2}+c(x)-\nu(x)+\int_{\sB_\circ} \frac{\abs{x+z}^\sigma}{\abs{x}^\sigma}
\nu(x,\D{z})\biggr)\chi(x)\,,\label{PT2.1Ab}
\end{align}
if $\cI$ satisfies \cref{P_growth}.
Now observe that for $x\in B_1^c$ there exist a positive constant $C$ such that 
\begin{equation}\label{PT2.1B}
\frac{\abs{x+z}^\sigma}{\abs{x}^\sigma} \,\le\, C(1+\abs{z}^\sigma)\,,
\quad \text{and}\quad \frac{\E^{\sigma\abs{x+z}}}{\E^{\sigma\abs{x}}}
\,\le\, \E^{\sigma\abs{z}}\,.
\end{equation}
Next, using \cref{PT2.1Aa,PT2.1Ab,PT2.1B}, and the fact that
$\supp\bigl(\nu(x,\cdot\,)\bigr)\subset\sB_\circ$,
 we deduce that for some a positive constant $C_1$
we have
\begin{equation}\label{ET2.1A}
\cI\chi(x) \le\, C_1\chi(x)\qquad \forall\,x\in B_{1}^c\,.
\end{equation}
Now set $\psi_n=\psi+\frac{1}{n}\chi$ and define $\kappa_n=\sup_{\Rd}\frac{u}{\psi_n}$.
Note that $\kappa_n>0$
for all $n\in\NN$ since $u$ is positive at some point.
Using the fact that
$\psi\ge u$ and the definition of $\kappa_n$ we can write $\kappa_n\le 1$,
and $\kappa_n\le\kappa_{n+1}$ for all $n\ge 1$. Moreover, since
$\sup_{\Rd} u < \infty$, we have
\begin{equation*}
\limsup_{|x|\rightarrow\infty}\,\frac{u(x)}{\psi_n(x)}\le\,0\,.
\end{equation*} 
Hence the supremum $\sup_\Rd \frac{u}{\psi_n}$
is attained, and there exists $x_n\in\Rd$ such that
$\kappa_n=\frac{u(x_n)}{\psi_n(x_n)}$.
Next, we estimate the term $\frac{\chi(x_n)}{n}$. Note that
\begin{equation*}
\frac{1}{\kappa_{2n}}\,\le\,\frac{\psi_{2n}(x_n)}{u(x_n)}
\,=\,\frac{1}{\kappa_n}-\frac{\chi(x_n)}{2n\,u(x_n)}\,,
\end{equation*}
which implies that
\begin{equation*}
\frac{\chi(x_n)}{n}\,\le\, 2\bigg(\frac{1}{\kappa_{n}}-\frac{1}{\kappa_{2n}}\bigg)u(x_n)
\,\le\, 2\bigg(\frac{1}{\kappa_{n}}-\frac{1}{\kappa_{2n}}\bigg)\psi(x_n)\,.
\end{equation*}
Hence, by continuity, for each $n\in\NN$ there exists a positive $\eta_n$ such that 
\begin{equation}\label{ET2.1B}
\frac{\chi(x)}{n} \,\le\, \bigg(\frac{1}{\kappa_{n}}-\frac{1}{\kappa_{2n}}\bigg)\psi(x)
\quad \text{in~} B_{\eta_n}(x_n)\,.
\end{equation}
On the other hand, using the linearity of $\cI$
together with \cref{ET2.1A,ET2.1B}, we obtain
\begin{equation*}
\cI\psi_n \,=\, \cI\psi+\frac{1}{n}\cI \chi 
\,\le\,  \biggl[-\lambda + C_1 \biggl(\frac{1}{\kappa_{n}}
-\frac{1}{\kappa_{2n}}\biggr) \biggr]\psi(x)\quad\forall\,x\in B_{\eta_n}(x_n)\,.
\end{equation*}
Since $\{\kappa_{n}\}_{n\in\NN}$ is a convergent sequence,
we can choose $n_\circ$ large enough so that 
\begin{equation}\label{ET2.1C}
\cI\psi_{n_\circ}\,<\,0\quad \text{in~} B_{\eta_{n_\circ}}(x_{n_\circ})\,.
\end{equation}
Note that for the nonnegative function $w=\kappa_{n_\circ}\psi_{n_\circ}-u$,
the following holds 
\begin{equation*}
\cI w \,=\, \kappa_{n_\circ} \cI \psi_{n_\circ} - \cI u \,<\, 0\quad
\text{in~} B_{\eta_{n_\circ}}(x_{n_\circ})\,,
\end{equation*}
which also implies that,
\begin{equation*}
\sA w \,\le\, \cI w < 0\quad \text{ in~} B_{\eta_{n_\circ}}(x_{n_\circ})\,.
\end{equation*} 
But $w(x_{n_\circ})=0$, and the strong maximum principle infers that $w\equiv 0$
in $B_{\eta_{n_\circ}}(x_{n_\circ})$.
Hence $\int_{\Rd} w(x+z)\nu(x,\D{z})=0$ in $B_{\eta_{n_\circ}}(x_{n_\circ})$.
Using these facts and \cref{ET2.1C}, we get 
\begin{equation*}
\cI u \,=\, \kappa_{n_\circ} \cI \psi_{n_\circ} < 0
\quad\text{in~} B_{\eta_{n_\circ}}(x_{n_\circ})\,,
\end{equation*}
which clearly contradicts the fact $\cI u\ge 0$ in $\Rd$.
Therefore we must have $u\le 0$.
 
(iii)\, As before, we consider $u\in\Sobl^{2,d}(\Rd)$ satisfying
\begin{equation*}
\cI u\ge 0 \quad \text{in~} \Rd\,,\quad \sup_{\Rd} u\,<\, \infty\,,
\end{equation*}
and since $\lambda^{\prime\prime}_1(\cI)>0$, we choose $\lambda>0$ and
$\psi\in\Sobl^{2,d}(\Rd)$ satisfying
$\cI \psi+\lambda\psi\le 0$ in $\Rd$ and $\inf_{\Rd} \psi >0$.
Scaling suitably, we may assume
that $\inf_{\Rd} \psi\ge 1$ which in turn, implies that
\begin{equation}\label{ET2.1D}
\cI \psi \,\le\, -\lambda \quad \text{in~} \Rd\,.
\end{equation}
If $\sup_{\Rd} u>0$, then $\kappa\df
\bigl(\sup_{\Rd}\frac{u}{\psi}\bigr)^{-1}\in (0, \infty)$.
Using \cref{ET2.1D}, we then
get
\begin{equation*}
-\lambda \,\ge\, \cI \Phi \,\ge\, \widetilde\cA\Phi \quad \text{in~} \Rd\,,
\end{equation*}
where $\Phi\df\psi-\kappa u$, and 
\begin{equation*}
\widetilde\cA \,\df\, \sum_{i, j} a_{ij}\partial_{x_ix_j}
+ \sum_{i} b_i(x)\partial_{x_i} - (c(x)-\nu(x))^{-}\,.
\end{equation*}
By the strong maximum principle, we have $\Phi>0$ in $\Rd$.
Then, applying \cite[Lemma~2.1(i)]{BHR-07}, it follows
that $\inf_{\Rd}\Phi>0$ which contradicts the fact that $\inf_{\Rd}\Phi=0$.
Thus $\kappa^{-1}=0$ and 
$\sup_{\Rd} u\le 0$. This completes the proof.
\end{proof}

As an application of \cref{T2.1},  we establish a relation between
$\lambda^\prime_1$ and  $\lambda^{\prime\prime}_1$ which extends the
result in \cite[Theorem~1.7]{Berestycki-15}.

\begin{theorem}\label{T2.2}
Suppose that $\lim_{n\to\infty} \lambda(\cI,B_n) =\lambda_1(\cI)$.
Then we have  $\lambda_1(\cI)\ge \lambda^\prime_1(\cI)$. 
Also, under the hypotheses of \cref{T2.1}\,\ttup{ii} or \ttup{iii},
we have $\lambda^\prime_1(\cI)\ge \lambda^{\prime\prime}_1(\cI)$.
\end{theorem}

\begin{proof}
The proof follows from the same argument as in \cite{Berestycki-15}.
We show that for any $\lambda>\lambda_1$ we have $\lambda\ge \lambda^\prime_1$.
Replacing $c$ by $c-\lambda$ we may
assume that $\lambda=0$. Since $\lim_{n\to\infty} \lambda(\cI, B_n) =\lambda_1$,
we can find a ball $\sB$ large enough such that the corresponding Dirichlet
eigenvalue $\lambda(\cI,\sB)$ is negative.
Let $\varphi_\sB$ be the corresponding principal eigenfunction.
We normalize the eigenfunction so that it satisfies
\begin{equation*}
\norm{\varphi_\sB}_{L^\infty(\sB)} \,=\,
\min\,\biggl\{1,\, \frac{-\lambda(\cI,\sB)}{\norm{c}_{L^\infty(\sB)}}\biggr\}\,.
\end{equation*}
Then for the equation $\cI f= c^{+} f^2$ we see that $\Bar u=1$ is a bounded
super-solution, and $\varphi_\sB$ is a sub-solution in $\Rd$ and also lies below $1$.
Now we can apply the monotone iteration
method (since the comparison principle holds for above operator)
to construct a positive solution which is bounded above by $1$.
This shows that $\lambda^\prime_1\le 0$. Hence the proof.

Next, we prove the second part of the theorem.
Suppose, on the contrary, that there exists $\lambda>0$ satisfying
$\lambda^{\prime\prime}_1(\cI)>\lambda>\lambda^{\prime}_1(\cI)$.
So from the definition of $\lambda^{\prime}_1(\cI)$ there exists a positive
$\phi\in\Sobl^{2,d}(\Rd)\cap L^\infty(\Rd)$ that satisfies
$\cI \phi+\lambda \phi\ge 0$ in $\Rd$.
But note that
$\lambda^{\prime\prime}_1(\cI + \lambda)=\lambda^{\prime\prime}_1(\cI)-\lambda>0$,
which means that
$\cI + \lambda$ satisfies the maximum principle in \cref{T2.1},
and this contradicts the existence of such a function $\phi$.
This completes the proof.
\end{proof}

The following observation is used in the sequel.

\begin{lemma}\label{L2.1}
For any domain $D$, it holds that $-\sup_{x\in D}c(x)\le \lambda^{\prime\prime}_1(\cI,D)$.
\end{lemma}

\begin{proof}
Suppose that this is not true.
Then, there exists $\lambda$ such that
$-\sup_{D}c>\lambda > \lambda^{\prime\prime}_1(\cI,D)$.
Now consider $\psi\equiv 1$.
We deduce that $\cI(1)+\lambda \le \sup_{D} c+ \lambda<0$.
Hence, the definition of $\lambda^{\prime\prime}_1(\cI,D)$ implies
that $\lambda^{\prime\prime}_1(\cI,D)\ge \lambda$, and this contradicts the
original hypothesis.
\end{proof}

We also need the following existence result in the exterior domain.

\begin{lemma}\label{L2.2}
Let $\sK$ be a compact domain in $\Rd$ with smooth boundary.
Let  $\Bar{v},\underline{v}\in\Sobl^{2,d}(\sK^c)\cap\cC(\Rd)$,
with $\underline{v}\le \Bar{v}$ in $\Rd$, satisfy
$\cI \Bar{v}\le f(x, \Bar{v})$ in $\sK^c$
for some function $f:\Rd\times\RR\to\RR$, 
and $\underline{v}\in\cC^2(\Rd)$ is a subsolution of $\cI u=f$ in $\sK^c$.
Assume that $f(x, \cdot)$ is locally Lipschitz, locally uniformly in $x$,
and $f$ is locally bounded.
Then there exists a solution $u\in\Sobl^{2,p}(\sK^c)\cap \cC(\Rd)$, $p>d$,
of $\cI u=f$ in $\sK^c$ satisfying  
$\underline{v}\le u \le \Bar{v}$.
\end{lemma}

\begin{proof}
 Let $\{\Omega_n\, :\,  n\in\NN\}$ be an increasing sequence of bounded domains with smooth boundary
 which cover $\sK^c$.
  Consider a
 solution $u_n\in\Sobl^{2,p}(\Omega_n)\cap\cC(\Rd)$ satisfying
\begin{equation}\label{EL2.2A}
\begin{split}
\cI u_n &\,=\, f(x,u_n) \quad \text{in~} \Omega_n\,,\\
u_n&\,=\,\underline{v} \quad \text{in~} \Omega_n^c\,.
\end{split}
\end{equation}
Existence of $u_n$ follows from a monotone iteration argument.
For instance, replacing $f$ by $f-\cI\underline{v}$ we may assume
that $\underline{v}=0$ and $\Bar{v}\ge 0$.
Choose $\theta$ large enough so that
$\lambda_{\Omega_n}(\cI-\theta)>0$, $\theta\ge\norm{c}_{L^{\infty}(\Omega_n)}$, and 
\begin{equation*}
\babs{f(x, s_1)-f(x, s_2)} \,\le\, \theta \abs{s_1-s_2}\quad\forall\,x\in\Omega_n\,,
\end{equation*}
with $s_1,s_2\in\bigl[\inf_{\Omega_n}\underline{v},\sup_{\Omega_n}\underline{v}\bigr]$.
Now, we let $\xi_1=0=\underline{v}$, and applying \cite[Theorem~6.3]{AB19},
we define the sequence
$\{\xi_{i}\}_{i\in\NN}\subset\Sobl^{2,p}(\Omega_n)\cap\cC(\Rd)$ 
via the recursion
\begin{equation*}
\cI\xi_{i+1}-\theta\xi_{i+1} \,=\, f(x,\xi_i)-\theta\xi_i
\quad \text{in~} \Omega_n\,,\quad \xi_{i+1}\,=\,0\quad \text{in~}\Omega^c_n\,.
\end{equation*}
Applying \cref{CA1} 
we obtain that $0\le \xi_1\le\xi_2\le \dotsb\le \Bar{v}$ in $\Omega_n$.
Then using \cite[Lemma~6.1]{AB19},
and  a standard elliptic regularity estimate,
 we pass to the limit $\xi_i\to u_n$ as $i\to\infty$, to get
a solution of \cref{EL2.2A}. We also have $\underline{v}\le u_n\le \Bar{v}$.

Defining
\begin{equation*}
\sJ_n(x)=\int_{\Rd} u_n(x+y)\nu(x,\D{y})\,,
\end{equation*}
we rewrite \cref{EL2.2A} as
\begin{equation*}
\sA u_n \,=\, -\sJ_n(x)+ f(x, u_n)\quad \text{in~} \Omega_n\,.
\end{equation*}
Since $\underline{v}\le u_n\le \Bar{v}$, it follows that $\sJ_n(x)$ and $f(x, u_n)$ are
locally uniformly bounded, and thus
$\{u_n\}$ is  bounded in $\Sobl^{2,p}(\sK^c)$  for $p>d$.
The behaviour of $u_n$ near the 
boundary $\partial\sK$ can also be uniformly controlled (cf. \cite[Lemma~6.1]{AB19}).
Thus we can extract 
a subsequence of $u_n$ converging to a function
$u\in\Sobl^{2,p}(\sK^c)\cap\cC(\Rd)$. Finally passing the limit in 
\cref{EL2.2A} we have desired result.
\end{proof}

The theorem which follows is used later to establish equality of eigenvalues.
This result
should be compared with \cite[Theorem~7.6]{Berestycki-15}.

\begin{theorem}\label{T2.3}
It holds that 
\begin{equation*}
\lambda^{\prime\prime}_1(\cI) \,=\,
\min\,\Bigl\{\lambda_1(\cI),\lim_{r\rightarrow\infty}
\lambda^{\prime\prime}(\cI,\Bar{B}_{r}^c)\Bigr\}\,,
\end{equation*}
where
\begin{equation}\label{ET2.3A}
\begin{aligned}
\lambda^{\prime\prime}(\cI, D) &\,\df\, \sup\,\Bigl\{\lambda\in\RR\,\colon
\exists\, \phi\in\Sobl^{2,d}(\Rd)\cap\cC(\Rd)  \\[5pt]
&\mspace{220mu} \mathrm{~satisfying~}
\cI \phi+\lambda \phi \,\le\, 0
\mathrm{~in~}D
 \mathrm{~and~} \inf_{\Rd}\phi>0 \Bigr\}\,.
\end{aligned}
\end{equation}
\end{theorem}

\begin{proof}
By the monotonicity property of $r\mapsto\lambda^{\prime\prime}(\cI,\Bar{B}_{r}^c)$, we have 
\begin{equation*}
\lambda^{\prime\prime}_1(\cI)\,\le\, \lim_{r\rightarrow\infty}\,
\lambda^{\prime\prime}(\cI,\Bar{B}_{r}^c)\,.
\end{equation*}
Using this fact and  the definition of the eigenvalues we can immediately write 
\begin{equation*}
\lambda^{\prime\prime}_1(\cI)\,\le\,\min\Bigl\{\lambda_1(\cI),
\lim_{r\rightarrow\infty}\lambda^{\prime\prime}(\cI,\Bar{B}_{r}^c)\Bigr\}\,.
\end{equation*}
To establish the reverse inequality, we show that for any
$\lambda <\min\bigl\{\lambda_1(\cI),\,\lim_{r\rightarrow\infty}
\lambda^{\prime\prime}(\cI,\Bar{B}_{r}^c)\bigr\}$
we also have $\lambda^{\prime\prime}_1(\cI)\ge \lambda$.
By hypothesis, we can find a positive number $R$ such that
$\lambda<\lambda^{\prime\prime}(\cI,\Bar{B}_{R}^c)$.
Thus there exists
$\eta\in\Sobl^{2,d}(\Bar{B}_{R}^c)\cap\cC(\Rd)$ satisfying
$\cI\eta+\lambda\eta\le 0$ in $\Bar{B}_{R}^c$ and $\inf_{\Rd}\eta>0$. 
After multiplying by a suitable constant we may assume that
$\inf_{\Rd}\eta\ge  2$.
Choose a Lipschitz continuous
function $f:\RR\to(-\infty, 0]$
satisfying $f(1)=-1$, $-1\leq f(t)\leq 0$ for $t\in [1,2]$ and $f(t)=0$ for $t\ge 2$.
Then $\Bar{u}=\eta$ is a  supersolution to 
\begin{equation*}
\cI u + \lambda u\,=\, \abs{c(x)+\lambda}\,f\bigl(u(x)\bigr)
\quad \text{in~} \Bar{B}_{R}^c\,,
\end{equation*}
and $\underline{u}=1$ is a subsolution.
By \cref{L2.2}, there exists 
$\phi\in\Sobl^{2,p}(\Bar{B}_{R}^c)\cap\cC(\Rd)$, $p>d$, 
satisfying $\inf_{\Rd}\phi\ge 1$ and
\begin{equation*}
\cI\phi+\lambda\phi \,\le\, 0 \quad \text{in\ } \Bar{B}_{R}^c\,.
\end{equation*} 

By Morrey's inequality we can see that $\phi\in\cC^1(\Bar{B}_{R+1}^c)$.
Since $\lambda < \lambda_1(\cI)$, there exist $\varepsilon_1>0$ and a positive function
$\psi\in\Sobl^{2,d}(\Rd)$ satisfying
\begin{equation*}
\cI\psi + (\lambda+\varepsilon_1)\psi\,\le\, 0\quad \text{in~} \Rd\,.
\end{equation*}
We choose a nonnegative function $\chi\in \cC^2(\Rd)$, taking values in $[0, 1]$,
such that $\chi=0$ in $B_{R+1}$ and $\chi=1$ in $B^c_{R+2}$.
Let $u\df\psi+\varepsilon\chi\phi$, with $\epsilon$ a positive constant
to be chosen later.
Recall from \cref{A1.1}\,(iii) that there exists a compact set $K_1$ such that 
$\supp\bigl(\nu(x,\cdot\,)\bigr)\subset K_1$ for all $x\in B_{R+2}$.

We continue with some estimates of $u$ on a  partition  of $\Rd$.
For $x\in B_{R+1}$ we have 
\begin{equation}\label{ET2.3B}
\cI u + \lambda u  \,\le\, -\varepsilon_1\psi+\varepsilon\int_{K_1}\chi(x+z)\phi(x+z)
\nu(x,{\rm d}z)\,\le\, -\varepsilon_1\psi+\varepsilon\kappa_1\,,
\end{equation}
where $\kappa_1$ depends upon the bounds of $\phi$ in the compact set $B_{R+1}+K_1$
and the measure of $K_1$.
If we consider the annular region $B_{R+2}\setminus B_{R+1}$,
then for all $x\in B_{R+2}\setminus B_{R+1}$, it holds that 
\begin{equation}\label{ET2.3C}
\begin{aligned}
\cI u + \lambda u & \,\le\, -\varepsilon_1\psi+\varepsilon\big[\cI(\chi\phi)
+\lambda\chi\phi\big]\\ 
& \,\le\,  -\varepsilon_1\psi+ \varepsilon\chi(\cI+\lambda)\phi+\varepsilon
\bigl[2a_{ij}\partial_{i}\chi\partial_{j}\phi+(a_{ij}\partial_{ij}\chi
+b_{i}\partial_{i}\chi)\phi\bigr] \\
&\mspace{100mu}
+\varepsilon\int_{\Rd}\big(\chi(x+z)-\chi(x)\big)\phi(x+z)\nu(x,{\rm d}z)\\
& \,\le\, -\varepsilon_1\psi + \varepsilon \kappa_2 + \varepsilon\int_{K_1}
\bigl(\chi(x+z)-\chi(x)\bigr)\phi(x+z)\nu(x,{\rm d}z)\\
&\,\le\, -\varepsilon_1\psi + \kappa_3\varepsilon\,,
\end{aligned}
\end{equation}
where $\kappa_2$ and $\kappa_3$ are some constants depending on the
bounds of the coefficients in the set 
$B_{R+2}\setminus B_{R+1}$.
Now choosing $\varepsilon$ small enough,
we can make the right-hand side of \cref{ET2.3B} and \cref{ET2.3C} non-positive.
Finally, when $x\in B_{R+2}^c$, we have 
\begin{equation*}
\cI u + \lambda u  \,\le\, -\varepsilon_1\psi+\varepsilon(\cI\phi+\lambda\phi)
+\varepsilon\int_{\Rd}\bigl(\chi(x+z)-1\bigr)\phi(x+z)\nu(x,{\rm d}z)\,\le\, 0\,.
\end{equation*}
Combining all the above cases, we deduce that $\cI u + \lambda u  \le 0$ in $\Rd$.
Hence, from the definition of $\lambda^{\prime\prime}_1(\cI)$, it is evident that
$\lambda^{\prime\prime}_1(\cI)\ge \lambda$, which completes the proof.
\end{proof}

\begin{remark}
Suppose $\sup_{D} c<\infty$.
Then the admissible test functions in the definition \cref{ET2.3A} can be restricted to
the class satisfying $\inf_D\phi>0$. That is, 
\begin{equation*}
\lambda^{\prime\prime}(\cI, D) \,=\,
\sup\,\Bigl\{\lambda\in\RR\,\colon \exists\,  \phi\in\Sobl^{2,d}(D)\cap\cC(\Rd),
\text{~satisfying~} \cI \phi + \lambda \phi \,\le\, 0 \text{~in~} D
\text{~and~} \inf_{D}\phi>0 \Bigr\}\,.
\end{equation*}
To see this consider any $\phi\in\Sobl^{2,d}(D)\cap\cC(\Rd)$ with $\inf_{D}\phi>0$ and 
\begin{equation*}\cI \phi + \lambda \phi \,\le\, 0\; \text{in~} D\,.\end{equation*}
For some $\varepsilon>0$, define $\psi_\varepsilon=\phi+\varepsilon$ and note that
for $x\in D$,
\begin{equation*}
\cI \psi_\varepsilon + \lambda \psi_\varepsilon
 \,\le\,  (c(x)+\lambda) \varepsilon \,\le\,\varepsilon
 \Bigl(\lambda+\sup_{D}\,c(x)\Bigr)^+\Bigl(\inf_{D}\,\phi\Bigr)^{-1}\psi_\varepsilon\,.
\end{equation*}
Since $\varepsilon$ is arbitrary, we get \cref{ET2.3A}.
\end{remark}

Our next result is an extension of \cite[Proposition~1.11]{Berestycki-15}.
It shows that if the potential $c$ is negative at infinity, then the principal
eigenvalue $\lambda_1(\cI)$ characterizes the validity of the maximum principle.

\begin{theorem}
Suppose that $\zeta\df\displaystyle\limsup_{\abs{x}\to\infty}\, c(x)<0$. Then
the following hold:
\begin{itemize}
 \item[(i)]
Under the hypotheses in \cref{T2.1}\,\ttup{ii} or \ttup{iii},
the maximum principle holds if $\lambda_1(\cI)>0$.
 \item[(ii)]
 If $\lim_{n\to\infty} \lambda(\cI,B_n) =\lambda_1(\cI)$ and
the maximum principle holds, then $\lambda_1(\cI)>0$.
 \end{itemize}
\end{theorem}
\begin{proof}
(i)\,
Suppose that $\lambda_1(\cI)>0$. Then applying \cref{L2.1,T2.3}
we deduce that $\lambda^{\prime\prime}(\cI)>0$,
which implies by \cref{T2.1} that $\cI$ satisfies the maximum principle. 

(ii)\, If $\cI$ satisfies the maximum principle,
then we have $\lambda^{\prime}(\cI)\ge 0$ by \cref{T2.1}\,(i),
and therefore, using \cref{T2.2} we deduce that $\lambda_1(\cI)\ge 0$.
Suppose that $\lambda_1(\cI)=0$. Taking $V=1$ we see that
$$\cI V = c(x) < \zeta/2\quad \text{outside a compact set.}$$
Hence by \cref{R1.4} and the proof of \cref{T1.4} we can construct a principal
eigenfunction $\varphi$ satisfying $\cI \varphi=0$ and $\varphi\leq V$ in $\Rd$.
This clearly, contradicts the maximum principle. Hence we must have 
$\lambda_1(\cI)>0$.
\end{proof}

Our next result should be compared with \cite[Theorem~1.9]{Berestycki-15}.

\begin{theorem}
$\lambda_1(\cI)=\lambda^{\prime\prime}_1(\cI)$ holds in each of the following cases:
\begin{itemize}
\item[(i)]
$\lambda_1(\cI-\gamma)=\lambda^{\prime\prime}_1(\cI-\gamma)$
for some nonnegative function $\gamma\in L^{\infty}(\Rd)$
satisfying $\underset{\abs{x}\to\infty}{\lim}\, \gamma(x)=0$;
\item[(ii)]
$\lambda_1(\cI)\le -\underset{\abs{x}\to\infty}{\lim\sup}\, c(x)$;
\item[(iii)]
$\norm{a}_{L^\infty(\Rd)}\le\Lambda$, $\lim_{\abs{x}\to\infty}b(x)=0$, and
\begin{equation*}
\forall\, r>0,\, \forall\, \beta<\,\limsup_{\abs{x}\to\infty} c(x),\, \exists\;
B_r(x_0)\; \text{satisfying~} \inf_{B_r(x_0)} (c(x)-\nu(x))>\beta\,;
\end{equation*}
\item[(iv)]
There exists $V\in \cC^2(\Rd)$ which is bounded from below away from $0$, and
satisfies
$\cI V +\lambda_1(\cI) V\le 0$ outside some compact set.
\end{itemize}
\end{theorem}
\begin{proof}
Hypotheses (i) and (ii) are exactly the same as in
\cite[Theorem~1.9(ii)-(iii)]{Berestycki-15}
which only uses \cref{T2.3}.

Next we show that (iii) $\Rightarrow$ (ii). Note that it is enough to demonstrate
that 
if $\sigma<\limsup_{|x|\rightarrow\infty}\, c(x)$ then $\lambda_1(\cI)\le -\sigma$.
Consider the function 
\begin{equation*}
\psi(x) \,=\, \exp\left(-\frac{1}{1-|\varepsilon x|^2}\right)
\text{\ on\ } B_{\nicefrac{1}{\varepsilon}}\,, \quad\text{ and\ }\psi(x)=0
\text{\ on\ } B_{\nicefrac{1}{\varepsilon}}^c\,,
\end{equation*}
for an appropriate positive constant $\varepsilon$ to be chosen later.
An easy calculation shows that
\begin{align*}
D_{x_i}\psi &\,=\,  \frac{-2\varepsilon^2x_i}{\bigl(1-|\varepsilon x|^2\bigr)^2}\psi
\intertext{and}
D_{x_ix_j}\psi &\,=\, \Biggl(\frac{4\varepsilon^4}{\bigl(1-|\varepsilon x|^2\bigr)^4}
x_ix_j - \frac{2\varepsilon^2}{\bigl(1-|\varepsilon x|^2\bigr)^2}\delta_{ij}
-\frac{8\varepsilon^4}{\bigl(1-|\varepsilon x|^2\bigr)^3}x_ix_j\Biggr)\psi\,.
\end{align*}

For $x_0\in\Rd$, define $\phi(x)=\psi(x-x_0)$.
Suppose that we show that we can choose $\varepsilon$ and $x_0$ in such a way that
\begin{equation}\label{ET2.5A}
\cI \phi-\sigma\phi>0\quad \text{in~} B_{\nicefrac{1}{\varepsilon}}(x_0).
\end{equation}
Since the principal eigenvalues
$\lambda(\cI,D)$ and $\lambda^{\prime}(\cI,D)$,
the later defined by \cref{ETA2B},
are equal for a bounded domain $D$ (see Theorem~\ref{TA2}),
we see that 
\begin{equation*}
-\sigma \ge\, \lambda^{\prime}\bigl(\cI,B_{\nicefrac{1}{\varepsilon}}(x_0)\bigr) \,=\,
\lambda\bigl(\cI,B_{\nicefrac{1}{\varepsilon}}(x_0)\bigr) \,\ge\, \lambda_1(\cI)
\end{equation*}
by \cref{ET2.5A}.
Thus it remains to establish \cref{ET2.5A}.
From the calculations of $D_{x_i}\psi$ and $D_{x_ix_j}\psi$ we see that
\begin{equation}\label{ET2.5B}
\begin{aligned}
\cI\phi(x)-\sigma\phi(x) 
&  \,\ge\, \biggl(\frac{4\Lambda \varepsilon^2 |\varepsilon (x-x_0)|^2}
{(1-|\varepsilon (x-x_0)|^2)^4}
-\frac{2d\Lambda\varepsilon^2}{(1-|\varepsilon (x-x_0)|^2)^2}
-\frac{8\Lambda\varepsilon^2 |\varepsilon (x-x_0)|^2}
{(1-|\varepsilon (x-x_0)|^2)^3}\\[5pt]
&\mspace{200mu}
 - \frac{2\epsilon^2|x-x_0|b(x)}
 {(1-|\epsilon (x-x_0)|^2)^2}+c(x)-\nu(x)-\sigma\biggr)\phi(x)\,.
\end{aligned}
\end{equation}
Given $\varepsilon>0$ we first choose $R$ large enough such that $|b(x)|\le \varepsilon$
for $|x|\ge R$, and
then choose $x_0\in\Rd$ satisfying $|x_0|\ge R+2\varepsilon^{-1}$.
Furthermore, due to our
hypothesis, we can choose $x_0$ is such a fashion that 
\begin{equation}\label{ET2.5C}
\inf_{B_{\nicefrac{1}{\varepsilon}}(x_0)} (c(x)-\nu(x))\,>\,\sigma\,.
\end{equation}
Next, we estimate \cref{ET2.5B} in two steps.

\noindent{\bf Step 1.}
Suppose that $1-\delta<|\varepsilon(x-x_0)|^2<1$ where $\delta$ is a small
positive number such that
\begin{equation*}
4\Lambda  (1-\delta)-2d\Lambda\delta^2 -8\Lambda(1-\delta)\delta  - 2 \delta^2>0\,.
\end{equation*}
It then follows from \cref{ET2.5B} that
\begin{equation*}
\cI\phi-\sigma\phi 
\,\ge\, \frac{\varepsilon^2 \bigl(4\Lambda  (1-\delta)-2d\Lambda\delta^2
-8\Lambda(1-\delta)\delta- 2 \delta^2\bigr)}
{(1-|\varepsilon (x-x_0)|^2)^4}\,\phi
+ \bigl(c(x)-\nu(x)-\sigma \bigr)\phi\,.
\end{equation*}  
This proves \cref{ET2.5A} in the annulus
$1-\delta<|\varepsilon(x-x_0)|^2<1$.

\noindent{\bf Step 2.}
We are left with the region $0 \le|\varepsilon (x-x_0)|^2\le 1-\delta$,
where $\delta$ is as chosen in Step~1.
Here we deduce that
\begin{equation*}
\cI\phi-\sigma\phi 
\,\ge\, \biggl(\big( c(x)-\nu(x)-\sigma \big) -\frac{2d\Lambda\varepsilon^2}{\delta^2}
-\frac{8\Lambda(1-\delta)\varepsilon^2}{\delta^3}
- \frac{2 \varepsilon^2}{\delta^2}\biggr)\phi\,.
\end{equation*}
Using \cref{ET2.5C}, we can choose $\varepsilon$ small enough to make the
right-hand side of the preceding equation positive.
Combining the above steps we obtain \cref{ET2.5A}, which completes the proof.

For (iv) we again use \cref{T2.3}.
Choose a large $r$ such that $B^c_{r}\subset\sK^c$.
Then 
$\lambda_1(\cI)\le \lambda^{\prime\prime}(\cI,\Bar{B}_{r}^c)$
by the definition.
Now letting $r\rightarrow\infty$ and using Theorem \cref{T2.3},
we deduce that $\lambda^{\prime\prime}_1(\cI)=\lambda_1(\cI)$.
\end{proof}

In the remaining part of this article we discuss the simplicity
of the principal eigenvalue.
As well known from \cite[Proposition~8.1]{Berestycki-15},
$\lambda_1(\cI)$ need not be simple, in general.
The following definition is the extension of \emph{Agmon's minimal growth at infinity}
\cite{Agmon-83} criterion in the nonlocal situation.

\begin{definition}\label{GS}
A positive function $u\in\Sobl^{2,d}(\Rd)$ satisfying $\cI u=0$ in $\Rd$
is said to be a solution of minimal growth at infinity if for any $\rho>0$
and any positive function $v\in\Sobl^{2,d}(B_{\rho}^c)\cap\cC(\Rd)$ satisfying
$\cI v\le 0$ in $B_{\rho}^c$, there exist constants
$R\ge\rho$ and $k>0$ such that $ku\le v$ in $B_{R}^c$.
\end{definition}

\begin{theorem}\label{T2.6}
Let $u\in\Sobl^{2,d}(\Rd)$ be a positive solution of minimal
growth at infinity of the equation $\cI u=0$ in $\Rd$.
Then, for any positive function $v\in\Sobl^{2,d}(\Rd)$ satisfying
$\cI v \le 0$ in $\Rd$, there exists $\kappa>0$ such that $v\equiv \kappa u$ in $\Rd$.
\end{theorem}

\begin{proof}
Define the quantity $\kappa=\inf_{\Rd}\frac{v}{u}$. Clearly $v-\kappa u\ge 0$.
First, we claim that  $v-\kappa u>0$ can not be positive in $\Rd$.
Suppose, on the contrary, that  $v-\kappa u>0$ in $\Rd$. Since $\cI(v-\kappa u)\leq 0$ in 
$\Rd$,
by the definition of solution of minimal growth at infinity of $u$,
there exist positive constants $R$ and $\kappa_1$ such that $\kappa_1 u\le v-\kappa u$
in $B_{R}^c$.
Note that $\cI u =0$ in $\Rd$ immediately gives $\lambda(\cI,B_{R+1})\ge 0$, and
the use of \cite[Corollary 2.1]{AB19} implies that $\lambda(\cI,B_{R})> 0$.
Applying \cref{CA1} we deduce that $\kappa_1 u\le v-\kappa u$ in $\Rd$ and this 
contradicts the definition of $\kappa$.
Thus $v-\kappa u$ must vanish somewhere in $\Rd$, proving the claim. Since
\begin{equation*}\sA (v-\kappa u)\le 0 \quad \text{in~} \Rd\,,\end{equation*}
 by the strong maximum principle we get $v=\kappa u$. This completes the proof.
\end{proof}

From the above result it is evident that only principal eigenfunctions can
have minimal growth at infinity.
Our next result is an extension of \cite[Proposition~8.4]{Berestycki-15}
which establishes a sufficient condition
for minimal growth at infinity of the eigenfunctions.

\begin{theorem}\label{T2.7}
Suppose that
\begin{equation*}
\lim_{r\rightarrow\infty}\,\lambda(\cI,\Bar{B}_{r}^c)\,>\,\lambda_1(\cI)\,,\qquad
\lim_{n\rightarrow\infty}\,\lambda(\cI,B_{n})\,=\,\lambda_1(\cI)\,,
\end{equation*}
 and 
$\supp\bigl(\nu(x,\cdot\,)\bigr)\subset \sB$ for some ball $\sB$,
for all $x\in\Rd$.
Then there exists an eigenfunction $\psi$ for $\lambda_1(\cI)$
which has minimal growth at infinity.
In particular,
 $\lambda_1(\cI)$ is simple in the class of positive functions.
 \end{theorem}

\begin{proof}
Without loss of generality we can assume $\lambda_1(\cI)=0$.
In view of \cref{T2.6}, it is enough to
show that there exists a principal eigenfunction with minimal growth at infinity.
For every $n\in\NN$, 
let $\psi_n\in\Sobl^{2,d}(B_{n})\cap\cC(\Rd)$ be an eigenfunction associated
with $\lambda(\cI,B_{n})$ i.e.
\begin{equation*}
\cI\psi_n \,=\, -\lambda(\cI,B_{n})\psi_n \text{ in }B_{n}\,,
\quad \psi_n>0 \text{\ in\ } B_{n} \,,
\text{ and }\, \psi_n=0 \text{\ in\ }B_{n}^c\,.
\end{equation*}
By hypothesis, there exist a large positive number $R$ and
$n_\circ\in\NN$ with $n_{\circ}>R$, such that for all $n\ge n_{\circ}$,
we have $\lambda(\cI,\Bar{B}_{R}^c)>\lambda>\lambda(\cI,B_{n})>0$, for some $\lambda>0$.

From the definition of $\lambda(\cI,\Bar{B}_{R}^c)$ there exists a nonnegative function
$\phi\in\Sobl^{2,d}(\Bar{B}_{R}^c)\cap\cC(\Rd)$ satisfying,
\begin{equation*}
\cI\phi \,\le\, -\lambda\phi \text{\ in\ }\Bar{B}_{R}^c\,,
\quad \phi>0 \text{\ in\ }\Bar{B}_{R}^c.
\end{equation*}
Let $\chi:\Rd\to [0, 1]$ be a $\cC^2$ cut-off function satisfying $\chi=1$ in $B_R$ and 
$\chi=0$ in $B_{R+1}^c$. We define $\varphi=\chi+\phi$.
Note that $\varphi>0$ in $\Rd$ and
\begin{equation*}
\cI\varphi \,\le\, -\lambda\varphi \text{\ in\ }\Bar{B}_{R_1}^c\,,
\quad \varphi>0 \text{\ in\ }\Rd \,,
\end{equation*}
for some large $R_1$ satisfying $\sB\cap(\sB-x)=\varnothing$
for all $x\in \Bar{B}_{R_1}^c$.

Fix any $n> \max\{R_1,n_\circ\}$. Let $\kappa_n>0$ such that
$\kappa_{n}\psi_n\le \varphi$ in $B_n$ and $\kappa_{n}\psi_n$ touches $\varphi$
at some point in $B_{n}$.
We claim that $\kappa_{n}\psi_n$ has to touch $\varphi$ inside $B_{R_1}$.
Note that in $\Bar{B}_{R_1}^c\cap B_{n}$ we have, 
\begin{equation*}
\sA(\varphi-\kappa_{n}\psi_n) \,\le\, \cI(\varphi-\kappa_{n}\psi_n)
\,\le\, -\lambda\varphi+\lambda(\cI, B_n)\kappa_{n}\psi_n
\,\le\,  (-\lambda+\lambda(\cI, B_n))\varphi \,<\, 0\,.
\end{equation*}
If $\kappa_{n}\psi_n$ touches $\varphi$ outside $B_{R}$, then applying the
strong maximum principle we have $\kappa_{n}\psi_n\equiv\varphi$
in $\Bar{B}_{R_1}^c\cap B_{n}$. But since $\psi_n=0$ on 
$\partial B_{n}$, we
must have $\varphi=0$ on $\partial B_{n}$  which is not possible. 
Normalizing, we work with the eigenfunction $\varphi_n\df\kappa_{n}\psi_n$ instead of
$\psi_n$.
Exploiting a similar method as in \cref{T1.4} and \cref{R1.4} we can show that
$\varphi_n$ converges along some subsequence, to a positive eigenfunction
$\psi\in \Sobl^{2,p}(\Rd)$, $p>d$, associated to the eigenvalue $\lambda_1(\cI)=0$. 

Now let $v$ be any positive supersolution as in \cref{GS} satisfying $\cI v\le 0$
in $B_{\rho}^c$.
Fix some $R^\prime>\max\{R_1, \rho\}$.
We scale $v$ by multiplying it with a positive constant $\kappa$
in such a way that $\sup_n\sup_{B_{R^\prime}}\varphi_n\le v$ in $B_{R^\prime}$.
Then to complete the proof it
is enough to show that $\psi\le v$ in $\Rd$.
For $\varepsilon>0$ we define $\zeta_\varepsilon=v+\varepsilon \varphi$ and we note that
$\cI \zeta_\varepsilon\le -\lambda\varepsilon\varphi$ in $B^c_{R^\prime}$. 
Hence in $\Bar{B}_{R^\prime}^c\cap B_{n}$ we have 
\begin{equation*}
\cI(\zeta_\varepsilon-\varphi_n) \,\le\, -\varepsilon\lambda\varphi+
\lambda(\cI, B_n)\varphi_n
\,\le\, (-\varepsilon\lambda + \lambda(\cI, B_n)\varphi \,<\, 0
\end{equation*}
for all large enough $n$, since $\lim_{n\to\infty}
\lambda(\cI, B_n)=\lambda_1(\cI)=0$.
Since
$\lambda\bigl(\cI,\Bar{B}_{R^\prime}^c\cap B_{n}\bigr)>\lambda_1(\cI)=0$
by \cref{CA1}, we have
$\psi_n\le \zeta_\varepsilon$ in $\Rd$.
Now we first let $n\to\infty$, and then $\varepsilon\to 0$, to
conclude that $\psi\le v$ in $\Rd$.
\end{proof}

\begin{remark}
\Cref{T2.7} should be compared with the second assertion in \cref{T1.4}.
At first sight, \cref{T2.7} seems to be stronger since
hypotheses (1), (2) and \cref{ET1.4A} of \cref{T1.4} are not enforced, but keep
in mind that $\lambda_1(\cI)>-\infty$ is a blanket assumption in
\cref{S2}.
\end{remark}

\begin{remark}
Suppose that  for some ball  $\sB_\circ$ we have 
$\supp\bigl(\nu(x,\cdot\,)\bigr)\subset\sB_\circ$ for all $x$ and, $\nu(x)$ is bounded.
Also, assume that the coefficients satisfy
\begin{equation*}
\norm{a}_{L^\infty(\Rd)}<\infty\,,\quad
\lim_{\abs{x}\to\infty}\frac{b(x)\cdot x}{\abs{x}}\,=\,\pm \infty\,,
\quad\text{and\ }\sup_{\Rd}\, c(x)\,<\,\infty\,.
\end{equation*}
This implies the condition
$\lim_{r\rightarrow\infty}\lambda(\cI,\Bar{B}_{r}^c)>\lambda_1(\cI)$ in \cref{T2.7}.
To see this, define the function
$\phi_r(x)\df\exp(\mp  \abs{x})$ in $\Bar{B}_{r}^c$,
$r>0$, where $\mp$ matches with the sign
$\pm$ in the hypothesis.
Now by direct calculation we obtain for $x\in \Bar{B}_{r}^c$,
\begin{align*}
(\cI+\lambda_1(\cI)+1)\phi_r(x) &\,=\,\Biggl(\frac{a_{ij}x_ix_j}{\abs{x}^2}
\mp\biggl(\frac{\text{Tr}(a_{ij})}{\abs{x}}-\frac{a_{ij}x_ix_j}{\abs{x}^3}
+\frac{b(x)\cdot x}{\abs{x}}\biggr)\\
&\mspace{100mu}+c(x)-\nu(x)+\lambda_1(\cI)+1
 +\int_{\sB_\circ}\frac{\E^{\mp\abs{x+z}}}{\E^{\mp\abs{x}}}
\nu(x,{\rm dz})\Biggr)\phi_r(x)\,.
\end{align*}
Now using the given hypotheses  and choosing large $r$, we deduce that
$(\cI+\lambda_1(\cI)+1)\phi_r<0$ for $x\in \Bar{B}_{r}^c$.
This implies that $\lambda(\cI,\Bar{B}_{r}^c)\ge\lambda_1(\cI)+1$ for
all large enough $r$ and completes the assertion.
\end{remark}

The result that follows establishes the equivalence between the minimal
growth  at infinity and
a certain monotonicity property of the principal eigenvalue.
We need the following notion of monotonicity from \cite[Section~2.2]{ABS-19}.
To express explicitly the dependency of the eigenvalue
of the potential $c$ we write the principal eigenvalue
$\lambda_1(\cI)$ as $\lambda_1(c)$.

\begin{definition}\label{D2.3}
We say $c\mapsto\lambda_1(c)$ is {\it strictly monotone on the right at $c$}
if for any non-zero, nonnegative
bounded function $h$ we have $\lambda_1(c+h)<\lambda_1(c)$.
\end{definition}

When $\nu=0$, it is shown in \cite[Theorem~2.1]{ABG-19} that Agmon's
minimal growth at infinity is
equivalent to the monotonicity property on the right.
The argument in \cite{ABG-19} is based on a probabilistic
method which uses the stochastic representation of the principal eigenfunction.
Our next result extends 
this equivalence for nonlocal operators, and also supplies a simpler proof.

\begin{theorem}\label{T3.8}
The following hold.
\begin{itemize}
\item[(i)]
Suppose that the principal eigenfunction $\psi$ satisfying
$\cI\psi+\lambda_1(c)\psi=0$ in $\Rd$ has minimal growth at infinity.
In addition, assume that
for every bounded $h\gneq 0$ there exists a positive supersolution of
$\cI + h+\lambda_1(c+h)$ in $\Rd$.
Then $\lambda_1(c)$ is strictly monotone on the right at $c$.
\item[(ii)]
Assume that $\lambda_1(c)$ is strictly monotone on the right at $c$ and there exists a 
positive function $\phi\in\Sobl^{2,d}(\sB^c)\cap\cC(\Rd)$ satisfying
\begin{equation*}
\cI\phi+\lambda_1(c)\phi \,\le\, 0\quad \text{in~} \sB^c\,,
\end{equation*}
where $\sB\subset\Rd$ is some ball.
Then there exists a principal eigenfunction for $\lambda_1(\cI)$, which
has minimal growth at infinity.
In particular,  $\lambda_1(\cI)$ is a simple eigenvalue
by \cref{T2.6}.
\end{itemize}
\end{theorem}

\begin{proof}
(i)\, Let $h\gneq 0$ be a bounded function.
From the definition of the principal eigenvalue it is evident that
$\lambda_1(c+h)\le\lambda_1(c)$.
Arguing by contradiction, suppose that
$\lambda_1(c+h)=\lambda_1(c)$. 
Consider a positive $\varphi\in\Sobl^{2,d}(\Rd)$ satisfying 
\begin{equation*}
\cI\varphi + h\varphi+\lambda_1(c+h)\varphi \,\le\, 0\quad \text{in~} \Rd\,.
\end{equation*}
Since $h\ge 0$ it follows from above that 
\begin{equation*}
\cI\varphi +\lambda_1(c)\varphi \,\le\, 0\quad \text{in~} \Rd\,.
\end{equation*}
Using \cref{T2.6} we then obtain $\varphi=\kappa\psi$ for some $\kappa>0$.
This of course, implies that
$h\varphi=0$ in $\Rd$ which contradicts the fact $h\neq 0$.
Hence we must have $\lambda_1(c+h)<\lambda_1(c)$.

(ii)\, We construct a principal eigenfunction with minimal growth at infinity. 
With no loss of generality we may assume that $\sB=B_1$, the unit ball centered at $0$.
 Let $f(x)=\Ind_{B_1}(x)$. 
Let $\varphi_n\in \Sobl^{2,p}(B_n)\cap \cC(\Rd)$ be the unique solution of 
\begin{equation*}
\cI \varphi_n + \lambda_1(c)\varphi_n \,=\,
-f\quad \text{in~} B_n\,, \quad\text{and\ \ } \varphi_n=0 \text{~on~} B_n^c\,.
\end{equation*}
Existence of $\varphi_n$ follows from \cref{TA3}, and
we have $\varphi_n>0$ in $B_n$.

We claim that $\beta_n\df\max_{\Bar{B}_1}\varphi_n\to\infty$ as $n\to\infty$.
To prove the claim, suppose that, on the contrary,
$\beta_{n_k}$ is bounded along some subsequence $\{n_k\}$. Define 
\begin{equation*}
\kappa_n \,=\,
\max\,\{t\,\colon \phi-t\varphi_n>0 \text{~in~} \Bar{B}_1\}\wedge 1\,.
\end{equation*}
It is evident that 
\begin{equation*}
\kappa_{n}\ge \left[\beta^{-1}_{n}\,\min_{\Bar{B}_1}\,\phi\right]\wedge 1\,.
\end{equation*}
Hence $\kappa_{n_k}\ge \hat\kappa>0$ for all $n_k$.
Letting $\psi_n=\kappa_n\varphi_n$, we obtain
\begin{equation*}
\cI \psi_n + \lambda_1(c)\psi_n \,=\, -\kappa_n f\quad \text{in~} B_n, \quad
\text{and\ \ } \psi_n=0\text{~on~} B_n^c.
\end{equation*}
Applying the comparison principle in \cref{CA1} in $B_n\setminus\Bar{B}_1$ we get
$\psi_n\le \phi$ in $\Rd$.
Next, applying an argument similar to \cref{T1.4} (see the arguments after \cref{ET1.4J})
we can extract a subsequence of $\{\psi_{n_k}\}$ and $\{\kappa_{n_k}\}$ converging 
to $\psi\in\Sobl^{2,p}(\Rd)$, $p>d$, and $\kappa^\prime>0$, respectively.
Furthermore, $\psi\ge 0$ and
\begin{equation}\label{ET2.8B}
\cI \psi + \lambda_1(c)\psi=-\kappa^\prime f\quad \text{in~} \Rd\,.
\end{equation}
Using the strong maximum principle we either have $\psi>0$ in $\Rd$ or $\psi\equiv 0$.
But $\psi\equiv 0$
is not possible since $\kappa^\prime f\neq 0$. Hence, $\psi>0$ in $\Rd$.
We write \cref{ET2.8B} as
\begin{equation*}
\cI\psi + \bigl(\kappa^\prime\Ind_{B_1}\psi^{-1}+\lambda_1(c)\bigr)\psi
\,=\, 0\quad \text{in~} \Rd\,,
\end{equation*}
which also implies $\lambda_1(c +\kappa^\prime\Ind_{B_1}\psi^{-1})\ge \lambda_1(c)$.
But this contradicts the assumption of strict monotonicity on the right.
This establishes the claim that
$\beta_n\df\max_{\Bar{B}_1}\varphi_n\to\infty$.
It also implies that $\kappa_n\to 0$ as $n\to\infty$,
and thus the constant $\kappa'$ in \cref{ET2.8B} equals $0$.
This shows that 
the subsequence $\psi_{n_k}$ converges to a positive $\psi\in\Sobl^{2,p}(\Rd)$
satisfying $\cI \psi+\lambda_1(c)\psi=0$ in $\Rd$.
Next we show that $\psi$ has minimal growth property at 
infinity.
Let $v\in \Sobl^{2,d}(\Bar{B}^c_\rho)\cap\cC(\Rd)$ be a positive function and 
\begin{equation*}
\cI v+\lambda_1(c)v \,\le\, 0 \quad \text{in~} \Bar{B}^c_\rho\,.
\end{equation*}
Without any loss of generality we may assume that $\rho\ge 1$.
Let $\kappa>0$ be such that $\kappa v-\psi_{n_k}\ge 0$ in $\Bar{B}_\rho$
for all $n_k\ge \rho$.
Now
\begin{equation*}
\cI (\kappa v-\psi_{n_k})+\lambda_1(c)(\kappa v-\psi_{n_k}) \,\le\, 0
\quad \text{in~} B_{n_k}\cap \Bar{B}^c_\rho\,.
\end{equation*}
Using \cref{CA1} we than have $\psi_{n_k}\le \kappa v$ in $\Rd$.
Letting $n_k\to\infty$ we then have
$\psi\le\kappa v$ in $\Rd$. This completes the proof.
\end{proof}

\appendix
\section{The Dirichlet Principal eigenvalue and its properties}\label{Appen}

Recall from \cref{E1.1,E1.2}
the operators 
\begin{equation*}
\cI f(x) \,=\, \trace\bigl(a(x) \grad^2 f\bigr) + b(x)\cdot\grad f(x)
+ c(x) f(x) + I[f, x]\,,
\end{equation*}
where 
\begin{equation*}
I[f, x] \,=\, \int_{\Rd} \bigl(f(x+z)-f(x)\bigr)\,\nu(x,\D{z})\,,
\end{equation*}
and
\begin{equation*}
\sA f(x) \,=\, \trace\bigl(a(x) \grad^2 f\bigr) + b(x)\cdot\grad f(x)
+ c(x) f(x) - \nu(x) f(x)\,,
\end{equation*}
with $\nu(x)=\nu(x,\Rd)$. Let $D$ be a bounded smooth domain. 
The following assumption on the coefficients is enforced throughout this section,
without further mention.

\begin{assumption}
The following hold.
\begin{itemize}
\item[(A1)]
$\nu(x,\cdot)$ is a nonnegative Borel measure and the map $x\mapsto \nu(x,\Rd)$
is locally bounded.
\item[(A2)]
The map $x\mapsto a(x)$ is continuous in $\Bar{D}$, and there exists a positive
constant $\upkappa$ such that $\upkappa I\le a(x)\le \upkappa^{-1} I$
for all $x\in\Bar{D}$, where $I$ denotes the identity matrix.
\item[(A3)]
$b\colon D\to\Rd$ and $c\colon D\to\RR$ are bounded.
\end{itemize}
\end{assumption}
By $\cC_{b,+}(\Rd)$ we denote the set of all bounded,
nonnegative continuous functions on $\Rd$.
Recall from \cref{E1.4}
that the Dirichlet principal eigenvalue of $\cI$ in $D$ is defined as follows:
\begin{equation*}
\lambda(\cI,D) \,\df\,\sup\,
\bigl\{\lambda\in\RR\,\colon \Psi(\lambda)\neq \varnothing \bigr\}\,,
\end{equation*}
where
\begin{equation*}
\Psi(\lambda)\,\df\,\bigl\{\psi\in\cC_{b, +}(\Rd)\cap\Sobl^{2,d}(D)\,\colon
\psi>0 \text{~in~} D, \text{~and~}
\cI\psi(x)+\lambda\psi \,\le\, 0 \text{~in~} D\bigr\}\,.
\end{equation*}
The following result is proved in \cite[Theorem~2.1]{AB19}.

\begin{theorem}\label{TA1}
There exists a unique $\psi_D\in\Sobl^{2,p}(D)\cap\cC_{b,+}(\Rd)$, $p>d$, satisfying
\begin{equation}\label{ETA1A}
\begin{split}
\cI \psi_D &\,=\, -\lambda(\cI,D)\, \psi_D\quad \text{in~} D\,, 
\\
\psi_D&\,=\,0\quad \text{in~} D^c\,,
\\
\psi_D(0)&\,=\,1\,,\quad \psi_D>0\quad \text{in~} D\,.
\end{split}
\end{equation}
\end{theorem}

We also have the following characterization of the principal eigenvalue.

\begin{theorem}\label{TA2}
It holds that
\begin{align}
\lambda(\cI,D) &\,=\, \sup\,\Bigl\{\lambda\in\RR\,\colon \exists\,
\psi\in\cC_{b}(\Rd)\cap\Sobl^{2,d}(D) \mathrm{~satisfying~} \inf_{\Rd}\psi>0\nonumber\\
&\mspace{420mu}\mathrm{~and~}
\cI\psi(x)+\lambda\psi \,\le\, 0\mathrm{~in~} D\Bigr\}\label{ETA2A}
\\
&\,=\, \inf\,\Bigl\{\lambda\in\RR\,\colon \exists\, \psi\in\cC_{b}(\Rd)\cap\Sobl^{2,d}(D)
\mathrm{~satisfying~} \sup_{D}\psi>0,\;\psi \,\le\, 0 \mathrm{~in~}  D^c,\nonumber\\
&\mspace{420mu}\mathrm{~and~} \cI\psi(x)+\lambda\psi\ge 0 \mathrm{~in~} D\Bigr\}\,.
\label{ETA2B}
\end{align}
\end{theorem}

\begin{proof}
Let $\lambda_1$ denote the rhs of \cref{ETA2A}.
It is then obvious from the definition that $\lambda(\cI,D)\ge\lambda_1$.
Let $D_n$ be a sequence of strictly decreasing domains of $\cC^2$ type that converges
to $D$. Then it follows from \cite[Theorem~2.2]{AB19} that
$\lambda(\cI,D_n)\to \lambda(\cI,D)$ as $n\to\infty$.
For a given $\varepsilon>0$ we fix $n$ large enough to satisfy
$\lambda(\cI,D)\le \lambda(\cI,D_n)+\varepsilon$.
Let $\psi_n$ be the Dirichlet principal eigenfunction corresponding to
$\lambda(\cI,D_n)$ and $\chi$ is a 
smooth cut-off function satisfying $\chi=0$ in $D$ and $\chi=1$ in $D_n^c$.
Let $\xi_\delta(x)=\psi_n(x)+\delta\chi(x)$.
It then follows that $\inf_{\Rd}\xi_\delta>0$ for every $\delta>0$.
In addition, for $\delta$ sufficiently small, we have
\begin{align*}
\cI\xi_\delta &\,=\, \cI\psi_n + \delta \int_{\Rd}\chi(x+z) \,\nu\bigl(x,\D(z)\bigr)\\
&\,=\,-\lambda(\cI,D_n)\xi_\delta + \delta \left[\lambda_n \chi(x)
+ \int_{\Rd}\chi(x+z)\, \nu\bigl(x,\D(z)\bigr)\right]\\
& \,\le\, (-\lambda(\cI,D) + \varepsilon)\xi_\delta + \delta \xi_\delta(x)
\sup_{x\in D}\, \frac{1}{\xi_\delta(x)}
\left[\lambda_n \chi(x) + \int_{\Rd}\chi(x+z)\, \nu\bigl(x,\D(z)\bigr)\right]\\[5pt]
& \,\le\, (-\lambda(\cI,D) + \varepsilon)\xi_\delta(x) + \varepsilon \xi_\delta(x)
\qquad\forall\,x\in D\,.
\end{align*}
Hence $\lambda_1\ge \lambda(\cI,D)-2\varepsilon$.
Since
$\varepsilon$ is arbitrary, we have $\lambda(\cI,D)=\lambda_1$, proving \cref{ETA2A}.

Next, let $\lambda_2$ denote the rhs of \cref{ETA2B}.
Note that the principal eigenfunction in \cref{ETA1A} is
a valid member of the admissible functions in \cref{ETA2B}.
Thus we have $\lambda(\cI,D)\ge \lambda_2$.
To establish the equality we show that for any $\mu<\lambda(\cI,D)$ there exists no
$\psi\in\cC_{b}(\Rd)\cap\Sobl^{2,d}(D)$
with $\sup_{D}\psi>0$ and $\psi\le 0$ in $D^c$ satisfying
\begin{equation}\label{ETA2C}
\cI \psi + \mu\psi \,\ge\, 0 \quad \text{in~} D\,.
\end{equation}
Suppose, on the contrary, that such $\psi$ exists.
Using the characterization in \cref{ETA2A} we can find 
$\varphi$ with $\inf_{\Rd}\varphi>0$ and 
\begin{equation*}
\cI \varphi + \lambda \varphi \,\le\, 0\quad \text{in~} D\,,
\end{equation*}
for some $\lambda\in (\mu, \lambda(\cI,D))$.
Define
\begin{equation*}
\kappa \,=\, \inf\,\{t>0\,\colon t\varphi-\psi>0 \text{~in~} D\}\,,
\quad \text{and}\quad 
\varphi_\kappa=\kappa\varphi-\psi\,.
\end{equation*}
Since $\sup_D\psi>0$, we have $\kappa>0$, and $\varphi_\kappa$ vanishes
at some point in $D$.
 Also, from \cref{ETA2C}, we have 
\begin{equation*}
\cI\varphi_\kappa  + \mu \varphi_\kappa \,\le\, \kappa(\mu-\lambda)\varphi \,\le\,
0\quad \text{in~} D\,.
\end{equation*}
This of course, implies that
\begin{equation*}
\trace\bigl(a(x) \grad^2 \varphi_\kappa\bigr) + b(x)\cdot\grad \varphi_\kappa(x)
- (c(x) -\nu(x)+\mu)^-\varphi_\kappa(x)
 \,\le\, 0 \quad \text{in~} D\,.
 \end{equation*}
Since $\varphi_\kappa$ attains its minimum $0$ inside $D$,
it follows from the strong maximum principle
that $\varphi_\kappa\equiv 0$ in $D$.
But this contradicts that fact that $\varphi_\kappa>0$ on $\partial D$. 
Hence there is no such $\psi$ satisfying \cref{ETA2C}.
Thus we have $\lambda_2\ge \lambda(\cI,D)$, giving us \cref{ETA2B}.
\end{proof}

Note that \cref{ETA2B} gives a refined maximum principle.

\begin{corollary}\label{CA1}
Suppose that $\lambda(\cI,D)>0$.
Then any $\psi\in\cC_{b}(\Rd)\cap\Sobl^{2,d}(D)$ satisfying
\begin{equation*}
\cI \psi\ge 0 \quad \text{in~} D\,, \quad \text{and}\quad \psi \,\le\, 0
\quad \text{in~} D^c\,,
\end{equation*}
is nonpositive in $\Rd$.
\end{corollary}

\cref{CA1} gives us the next existence result.

\begin{theorem}\label{TA3}
Suppose that $\lambda(\cI,D)>0$ and $f\in \cC(\Bar{D})$.
Then there exists a unique $u\in \cC_{b}(\Rd)\cap\Sobl^{2,d}(D)$ satisfying
\begin{equation*}
\begin{split}
\cI u &\,=\, f \quad \text{in~} D\,,\\
u&\,\,=\,0 \quad \text{in~} D^c\,.
\end{split}
\end{equation*}
Furthermore, if $f\lneq 0$, then $u>0$ in $D$.
\end{theorem}

\begin{proof}
Uniqueness follows from \cref{CA1}.
Existence follows from a standard monotone iteration method
(cf.\ \cite[Proposition~4.5]{QS08}).
The last conclusion follows from the strong maximum principle.
\end{proof}

\subsection*{Acknowledgement}
The authors are indebted to the referees for their careful reading and suggestions.
The research of Ari Arapostathis was supported
in part by the National Science Foundation through grant DMS-1715210, and
in part by the Army Research Office through grant W911NF-17-1-001.
The research of Anup Biswas was supported in part by DST-SERB grants EMR/2016/004810, MTR/2018/000028 and a SwarnaJayanti fellowship
SB/SJF/2020-21/03. Prasun Roychowdhury was supported in part by Council of Scientific \& Industrial Research (File no. 09/936(0182)/2017-EMR-I).



\end{document}